\documentclass[11pt,leqno]{amsart}
\usepackage{amsmath,amsfonts,amssymb,amscd,amsthm,amsbsy,upref,verbatim}

\voffset=-.5truein \numberwithin{equation}{section}
\numberwithin{equation}{section}
\newtheorem{thm}{Theorem}[section]
\newtheorem{lem}[thm]{Lemma}

\theoremstyle{definition}
\newtheorem{defn}[thm]{Definition}
\newtheorem{remark}[thm]{Remark}

\newtheorem*{claim}{Claim}

\def\real{{\mathbb R}}

\newcommand{\loce}{\equiv_L}

\newcommand{\R}{\mathbb{R}}
\newcommand{\Q}{\mathbb{Q}}

\newcommand{\N}{\mathbb{N}}

\newcommand{\id}{{\mbox{\rm id}}}

\def\int{{\mathbb Z}}

\def\ep{\varepsilon}
\newcommand{\sF}{\mathcal{F}}
\newcommand{\ww}{\omega^\omega}

\newcommand{\pqc}{\mathcal{B}_{\vec p, \vec q, \vec n}}
\newcommand{\snc}{\mathcal{B}_{\vec S,, \vec n}}

\newcommand{\bs}{\boldsymbol{\Sigma}}
\newcommand{\bp}{\boldsymbol{\Pi}}

\newcommand{\res}{\restriction}

\newcommand{\norm}{||}

\newcommand{\spn}{\text{span}}

\catcode`@=11 \@mparswitchfalse  
\newcounter{mnotecount}[section]

\newcommand{\mnote}[1]{} 

\begin{document}

\title[The complexity of uniform homeomorphism]{On the complexity of the uniform homeomorphism relation between separable Banach spaces}
\author{Su Gao, Steve Jackson, and B\"{u}nyamin Sar\i}
\address{Department of Mathematics, University of North Texas,
1155 Union Circle \#311430, Denton, TX 76203-5017}
\email{sgao,jackson,bunyamin@unt.edu}
\thanks{The first author acknowledges the partial support of the NSF grant
DMS-0501039 and a Templeton Foundation research grant for his
research. The third author acknowledges an NSF funded Young
Investigator Award of the Analysis and Probability workshop at the
Texas A\&M University in 2008.} \maketitle

\tableofcontents

\section{Introduction}

Recently, there has been a growing interest in understanding the
complexity of common analytic equivalence relations between
separable Banach spaces via the notion of Borel reducibility in
descriptive set theory (see \cite{Bo} \cite{FG} \cite{FLR}
\cite{FR1} \cite{FR2} \cite{Me}). In general, the notion of Borel
reducibility yields a hierarchy (though not linear) among
equivalence relations in terms of their relative complexity.

The most important relations between separable Banach spaces include
the isometry, the isomorphism, the equivalence of bases, and in
nonlinear theory, Lipschitz and uniform homeomorphisms. The exact
complexity of the first four relations has been completely
determined by recent work in the field. Using earlier work of Weaver
\cite{We}, Melleray \cite{Me} proved that the isometry between
separable Banach spaces is a universal orbit equivalence relation.
Rosendal \cite{Ro} studied the equivalence of bases and showed that
it is a complete $K_\sigma$ equivalence relation. Using the work of
Argyros and Dodos \cite{AD} on amalgamations of Banach spaces,
Ferenczi, Louveau, and Rosendal \cite{FLR} recently showed that the
isomorphism, the (complemented) biembeddability and the Lipschitz
equivalence between separable Banach spaces, as well as the
permutative equivalence of Schauder bases, are complete analytic
equivalence relations. The Borel reducibility among these
equivalence relations, as well as some other equivalence relations
we will be dealing with in this paper, is illustrated in
Figure~\ref{fig:ers} (see Section~\ref{sec:Pre} below for the
definitions of the equivalence relations). Note that in particular
the complete analytic equivalence relation $E_{{\bf\Sigma}^1_1}$ is
the most complex one in the Borel reducibility hierarchy of all
analytic equivalence relations.
\begin{figure}[h]
\begin{center}
\setlength{\unitlength}{.3cm}
\begin{picture}(20,25)(0,0)

\put(10,1){\circle*{.3}}
\put(12,0.2){\makebox(0,0)[b]{$\id(2^\omega)$}}
\put(10,1){\line(0,1){3}}
\put(10,4){\circle*{.3}}
\put(11,2.8){\makebox(0,0)[b]{$E_0$}}

\put(10,4){\line(-3,2){6}}
\put(4,8){\circle*{.3}}
\put(3,6.5){\makebox(0,0)[b]{$E_0^\omega$}}
\put(10,4){\line(-1,2){2}}
\put(8,8){\circle*{.3}}
\put(7.5,6.5){\makebox(0,0)[b]{$E_\infty$}}
\put(10,4){\line(1,2){2}}
\put(12,8){\circle*{.3}}
\put(12.5,6.5){\makebox(0,0)[b]{$\ell_1$}}
\put(10,4){\line(3,2){6}}
\put(16,8){\circle*{.3}}
\put(16.2,6.5){\makebox(0,0)[b]{$E_1$}}

\put(4,8){\line(0,1){6}}
\put(4,14){\circle*{.3}}
\put(3,13.8){\makebox(0,0)[b]{$=^+$}}
\put(4,14){\line(2,-3){4}}
\put(16,8){\line(0,1){6}}
\put(16,14){\circle*{.3}}
\put(17,13.2){\makebox(0,0)[b]{$\ell_\infty$}}
\put(21.2,13.2){\makebox(0,0)[b]{(equivalence}}
\put(21,12){\makebox(0,0){of bases)}}
\put(16,14){\line(-2,-3){4}}
\put(16,14){\line(-4,-3){8}}

\put(10,23){\line(2,-3){6}}
\put(10,23){\line(-2,-3){6}}
\put(10,23){\circle*{.3}}
\put(10.2,23){\makebox(0,0)[b]{$E_{{\bf\Sigma}^1_1}$}}
\put(15.5,23){\makebox(0,0)[b]{(isomorphism)}}

\put(12,8){\line(-2,3){6}}
\put(6,17){\circle*{.3}}
\put(5,17){\makebox(0,0)[b]{$E_G^\infty$}}
\put(9.5,17){\makebox(0,0)[b]{(isometry)}}

\end{picture}
\caption{\label{fig:ers}Equivalence relations characterizing the
complexity of
classification problems for Banach spaces.}
\end{center}
\end{figure}
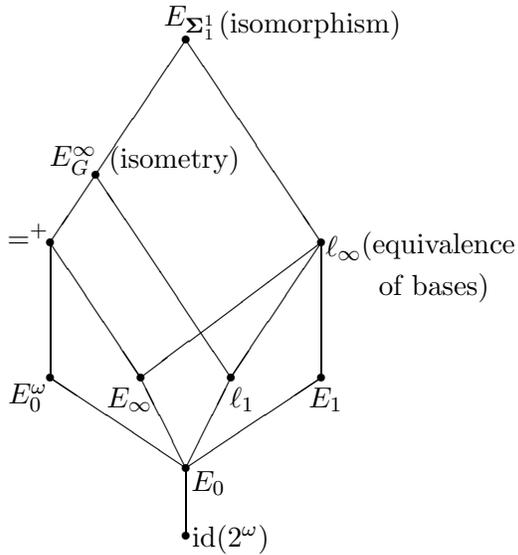

The main problem left open now is to determine the exact complexity
of the uniform homeomorphism between separable Banach spaces (see
Problem 23, \cite{FLR}). Recall that Banach spaces $X$ and $Y$ are
uniformly homeomorphic if there exists a uniformly continuous
bijection $\phi:X\to Y$ such that $\phi^{-1}$ is also uniformly
continuous. The understanding of the uniform homeomorphism relation
between Banach spaces, also known as the uniform classification of
Banach spaces, is in fact one of the main programs in the nonlinear
theory of Banach spaces.

Compared to the linear theory, results about the uniform
classification are significantly harder to prove, and their proofs
often use a combination of metric, geometric, and topological
arguments (for a good survey of methods and results see Chapter 10,
\cite{BL}). Moreover, previous efforts have been mostly focused on
the uniform classification of classical Banach spaces. For instance,
it is well-known that for $1\leq p\neq q< \infty$, $\ell_p$ and
$\ell_q$ are not uniformly homeomorphic (due to Ribe \cite{Ri1}).
Also, for $p\neq 2$ $\ell_p$ and $L_p$ are not uniformly
homeomorphic (due to Bourgain \cite{Bou} for $1\leq p<2$ and Gorelik
\cite{Go} for $2<p<\infty$). In fact, it turns out that for
$1<p<\infty$, the uniform structure of $\ell_p$ completely
determines its linear structure (a result due to Johnson,
Lindenstrauss, and Schechtman \cite{JLS}). This also generalizes to
certain finite sums of $\ell_p$ spaces.

From the point of view of descriptive set theory all previously
known results on the uniform classification give some lower bound
estimates for its complexity. However, these lower bounds are no
more complex than $\id(2^\omega)$, the least complicated one in
Figure~\ref{fig:ers}. In contrast to this, it is conceivable that
the uniform classification is as complex as Lipschitz and isomorphic
classifications, that is, it is $E_{{\bf\Sigma}^1_1}$. Thus there is
a huge gap between what was conjectured and what could be verified.

In this paper, we give a slightly improved lower bound for the
complexity of the uniform classification. We show that the complete
$K_\sigma$ equivalence relation (represented by the equivalence
relation $\ell_\infty$, to be defined in Section~\ref{sec:Pre}) is
Borel reducible to the uniform homeomorphism relation between
separable Banach spaces. As shown in Figure~\ref{fig:ers}, this in
particular implies that the uniform homeomorphism relation is at
least as complex as the equivalence of bases.

The study of the uniform classification is essentially devoted to
understanding what aspects of the linear structure of Banach spaces are
invariant under uniform homeomorphisms. As an important example, a
fundamental theorem of Ribe \cite{Ri1} asserts that the local
structure of finite-dimensional subspaces is such an invariant. The
proof of our main theorem is a straightforward application of this
theorem of Ribe. Moreover, we will isolate a concept of local
equivalence between separable Banach spaces in
Section~\ref{sec:PreB} and prove in Section~\ref{sec:Ribe} that it
is Borel bireducible with $\ell_\infty$. This means that the lower
bound we have reached for the complexity of the uniform
classification is the best possible with the consideration of local
structures.

We can now state the main theorems of this paper.

\begin{thm}\label{thm:main1} There exists a Borel family $\{ S_{\vec{x}}\,:\, \vec{x}\in
\R^\omega\}$ of separable Banach spaces such that the following are
equivalent for any $\vec{x},\vec{y}\in\R^\omega$:
\begin{enumerate}
\item[(a)] $\vec{x}-\vec{y}\in \ell_\infty$;
\item[(b)] $S_{\vec{x}}$ and $S_{\vec{y}}$ are uniformly homeomorphic;
\item[(c)] $S_{\vec{x}}$ and $S_{\vec{y}}$ are isomorphic;
\item[(d)] $S_{\vec{x}}$ and $S_{\vec{y}}$ are locally equivalent.
\end{enumerate}
\end{thm}

\begin{thm}\label{thm:main2} The local equivalence between separable Banach
spaces is Borel bireducible with $\ell_\infty$.
\end{thm}

Of course in general the local structure is not sufficient to
determine the uniform structure (for instance, as the results of
Bourgain and Gorelik mentioned above show). It is anticipated that
the complexity of the uniform classification is much more complex
than $\ell_\infty$. To verify this it would be enough to show that
the equivalence relation $E_0^\omega$ is Borel reducible to the
uniform homeomorphism relation. As Figure~\ref{fig:ers} suggests,
$E_0^\omega$ is in some sense the least complex equivalence relation
not Borel reducible to $\ell_\infty$.

In Section~\ref{sec:class} we generalize the construction in the
proof of Theorem~\ref{thm:main1} and consider a variety of classes
of separable Banach spaces with a similar construction scheme. The
uniform homeomorphism relations for these classes are all no more
complex than $\ell_\infty$. We then determine exactly what kind of
complexity the uniform homeomorphism relations on these classes can
achieve. It turns out that they can only be $\ell_\infty$, $E_1$,
$E_0$, or smooth.

Our constructions in Sections~\ref{sec:cons} and \ref{sec:class}
will yield only classes of separable Banach spaces for which the
uniform homeomorphism and the isomorphism relations coincide. In
general it is well-known that the uniform homeomorphism is a
genuinely coarser equivalence relation than the isomorphism (see,
for example, Section 10.4, \cite{BL}). Therefore it is of interest
to study the question how many different isomorphism classes a
single uniform homeomorphic class can contain. In
Section~\ref{sec:noniso} we prove the following related result.

\begin{thm}\label{thm:main3} There exists a Borel class ${\mathcal C}$
of mutually uniformly homeomorphic separable Banach spaces such that
the equality relation of countable sets of real numbers, denoted $=^+$,
is Borel reducible to the isomorphism relation on ${\mathcal C}$.
\end{thm}

The rest of the paper is organized as follows. In
Section~\ref{sec:Pre} we define all benchmark equivalence relations
relevant to our discussions in this paper and review the Borel
reducibility theory of equivalence relations on Polish spaces. In
Section~\ref{sec:PreB} we explain how to apply the framework of the
descriptive set theory of equivalence relations to the uniform
classification of separable Banach spaces. We also define the notion
of local equivalence and show that it is ${\bf\Sigma}^0_3$ in two
different codings of separable Banach spaces. In
Section~\ref{sec:cons} we give the construction of the family in
Theorem~\ref{thm:main1} and prove some basic properties. In
Section~\ref{sec:Ribe} we give the proofs of both
Theorems~\ref{thm:main1} and \ref{thm:main2}. We also generalize the
$\ell_\infty$ equivalence relation and prove a dichotomy theorem
characterizing its possible complexity. In Section~\ref{sec:class}
the construction of Section~\ref{sec:cons} is generalized and the
possible complexity of the uniform homeomorphism relations for the
resulting classes is completely determined. Finally in
Section~\ref{sec:noniso} we prove Theorem~\ref{thm:main3}.

{\it Acknowledgment.} We would like to thank Christian Rosendal for his careful 
reading of earlier versions of this paper and for making suggestions that improved 
the manuscript significantly. We also thank Pandelis Dodos and Greg Hjorth
 for useful discussions on the topics of this paper.

\section{\label{sec:Pre}Preliminaries on the Borel reducibility hierarchy}

In this section we review the Borel reducibility hierarchy of
analytic equivalence relations for the convenience of the reader. We
give the definitions of all equivalence relations mentioned in
Figure~\ref{fig:ers} and recall their characteristic properties. The
reader can find more details in the references provided below, or
see \cite{Ga}.

The descriptive set theory studies definable sets and relations on
Polish spaces. Recall that a {\it Polish} space is a separable and
completely metrizable topological space. Examples of Polish spaces
include $\omega={\mathbb N}$ with the discrete topology, ${\mathbb
R}$ with the usual topology, ${\mathbb R}^\omega$ with the product
topology, and the Cantor space $2^\omega=\{0,1\}^\omega$. The
simplest examples of definable subsets of a Polish space are the
Borel sets. Recall that the collection of all {\it Borel} sets is
the smallest $\sigma$-algebra containing all open sets. Thus all
Borel subsets of a Polish space can also be arranged in a hierarchy
according to their {\it descriptive complexity}. In this hierarchy
the simplest ones are open sets and closed sets. On the next level
we have the $F_\sigma$ sets and $G_\delta$ sets, which are
respectively countable unions of closed sets and countable
intersections of open sets. To continue, we call a set
${\bf\Sigma}^0_3$ if it is a countable union of $G_\delta$ sets, and
${\bf\Pi}^0_3$ if it is a countable intersection of $F_\sigma$ sets.
In general, one can define the classes ${\bf\Sigma}^0_\alpha$,
${\bf\Pi}^0_\alpha$ in the same fashion for all countable ordinals
$\alpha$. However, in this paper we will not deal with any set
beyond ${\bf\Sigma}^0_3$ and ${\bf\Pi}^0_3$.

It is well-known that ${\bf\Sigma}^0_3$ and ${\bf\Pi}^0_3$ are
distinct classes. To prove that a given subset $A$ of a Polish space
$X$ is not $\bs^0_3$, the usual method is to try and show that $A$
is {\it $\bp^0_3$-hard}, that is, given any $\bp^0_3$ subset $B$ of
$2^\omega$, there is a continuous function $f: 2^\omega\to X$ such
that $B=f^{-1}(A)$. If $A$ is $\bp^0_3$ and $\bp^0_3$-hard then it
is said to be {\it $\bp^0_3$-complete}. For more on this topic see
Section 22, \cite{Ke}.

The Borel structure thus given by a Polish topology is known as a
{\it standard Borel structure}. A Borel space (that is, a space with
a distinguished $\sigma$-algebra of subsets) is called a {\it
standard Borel space} if its Borel structure is standard, that is,
induced by an underlying Polish topology. A function $f$ between
Polish spaces (or standard Borel spaces) is {\it Borel} if
$f^{-1}(A)$ is Borel for any Borel set $A$. Any two uncountable
standard Borel space are Borel isomorphic to each other.

Other than the examples of Polish spaces mentioned above, we recall
another well-known example of a standard Borel space, the Effros
Borel space. Let $X$ be a Polish space and $F(X)$ be the hyperspace
of all closed subsets of $X$. The {\it Effros Borel structure} is
the Borel structure on $F(X)$ generated by basic Borel sets of the
form
$$ \{ F\in F(X)\,:\, F\cap U\neq \emptyset\} $$
for some open subset $U$ of $X$. A Polish topology generating the
Effros Borel structure was discovered by Beer \cite{Be}. We also
recall its definition. Let $d$ be a compatible complete metric on
$X$. For any $x\in X$ and $F\in F(X)$, define
$$ d(x,F)=\inf\{d(x,y)\,:\, y\in F\}. $$
Now consider the topology generated by all subbasic open sets of the form
$$ \{ F\in F(X)\,:\, d(x, F)<a\} \mbox{ or } \{ F\in F(X)\,:\, d(x, F)>a\} $$
for some $x\in X$ and $a\in{\mathbb R}$. This topology is known as
the Wijsman topology on $F(X)$; Hess \cite{He} observed that it
generates the Effros Borel structure, and Beer \cite{Be} proved that
it is Polish.

The next level of definable subsets beyond Borel subsets of a Polish
space consists of analytic ones. Recall that a subset of a Polish
space is {\it analytic} (or $\bs^1_1$) if it is the continuous image
of a Polish space. It is well-known that every Borel set is
analytic. For more information on Polish spaces, Borel and analytic
subsets, and Borel functions, see \cite{Ke}.

Let $X$ be a Polish space and $E$ an equivalence relation on $X$. We
say that $E$ is {\it analytic} if $E$ is an analytic subset of
$X\times X$. Similarly we also speak of {\it Borel} equivalence
relations, or even $F_\sigma$, $G_\delta$, ${\bf\Sigma}^0_3$,
${\bf\Pi}^0_3$ equivalence relations respectively.

The notion of Borel reducibility defined below is fundamental in the
theory of equivalence relations as it explores the relative {\it
structural complexity} of equivalence relations. Let $X, Y$ be
Polish spaces and $E, F$ be equivalence relations on $X, Y$,
respectively. We say that $E$ is \emph{Borel reducible} to $F$, and
denote it by $E\leq_B F$, if there is a Borel function $f:X\to Y$
such that for all $x_1, x_2\in X$,
$$x_1Ex_2 \iff f(x_1)Ff(x_2).$$
If $E\leq_B F$, then intuitively $E$ is \emph{no more complex than}
$F$, since any complete invariants for the $F$-equivalence classes
can be composed with $f$ to obtain complete invariants for the
$E$-equivalence classes. In the case of both $E\le_B F$ and $F\le_B
E$, then we denote $E\sim_B F$ and say that $E$ and $F$ are {\it
Borel bireducible}. If $E\sim_B F$, then intuitively $E$ and $F$
have {\it the same complexity}.

Next we define the benchmark equivalence relations in Figure~\ref{fig:ers}.

\begin{itemize}
\item[(1)] The equivalence relation $\id(2^\omega)$ is the identity
(or equality) relation on the Cantor space $2^\omega$, that is,
$(x,y)\in \id(2^\omega)$ iff $x=y$. Since all uncountable standard
Borel spaces are Borel isomorphic to each other, this relation is
Borel bireducible with the identity relation on any
uncountable Polish (or standard Borel) space. An equivalence relation
that is Borel reducible to $\id(2^\omega)$ is said to be {\it
smooth} or {\it concretely classifiable}, since it is possible to
assign a concrete real number as a complete invariant for each of
its equivalence classes.
\item[(2)] The equivalence relation $E_0$ is the eventual agreement
relation on $2^{\omega}$. In symbols, for $x=(x_n),y=(y_n)\in 2^\omega$,
$$ xE_0y \iff \exists m\in\omega\ \forall n\ge m\ x_n=y_n.$$
In the Borel reducibility hierarchy for Borel equivalence relations
$E_0$ is the minimum one beyond $\id(2^\omega)$ \cite{HKL}.

\item[(3)] The equivalence relation $E_1$ is the eventual agreement
relation for countable sequences of real numbers. In symbols, for
$\vec{x}=(x_n), \vec{y}=(y_n)\in \R^{\omega}$,
$$ \vec{x}E_1\vec{y} \iff \exists m\in\omega\ \forall n\ge m\ x_n=y_n.$$
It is easy to see that $E_0\le_B E_1$. In the definition of $E_1$
the space $\R$ can be replaced by the Cantor space $2^\omega$ or the Baire space $\omega^\omega$
without affecting the complexity of the resulting equivalence
relation, since $\R$ is Borel isomorphic to any uncountable Polish space. We will
use the alternate versions of the definition in this paper without further
elaborations. In the Borel reducibility hierarchy $E_1$ is an immediate successor of $E_0$ (\cite{KL}), that is, if $E\leq_B E_1$
then $E$ is Borel bireducible with either $E_1$ or $E_0$, or else $E$ is smooth.

\item[(4)] For $1\leq p\leq \infty$ the equivalence relation
$E_{\ell_p}$ is defined on $\R^\omega$ as follows: for
$\vec{x}=(x_n), \vec{y}=(y_n)\in \real^{\omega}$,
$$\vec{x}E_{\ell_p}\vec{y} \iff \vec{x}-\vec{y}\in \ell_p.$$
When there is no danger of confusion we simply use $\ell_p$ to
denote the equivalence relation $E_{\ell_p}$. Dougherty and Hjorth
\cite{DH} showed that for $1\leq p\leq q<\infty$,
$\ell_p\leq_B\ell_q$. The first author \cite{Gao} extended this to
include the case $q=\infty$. It is also known that $E_1\leq_B
\ell_\infty$ \cite{Gao} and $E_1\nleq_B \ell_p$ for $p<\infty$
\cite{KL}.

The equivalence relation $\ell_\infty$ is perhaps the most important
equivalence relation for this paper. Rosendal \cite{Ro} showed that
it is a complete $K_\sigma$ equivalence relation, that is, for any
equivalence relation $E$ on a Polish space $X$, if $E$
is $K_\sigma$ (that is, a countable union of compact subsets of $X^2$), then
$E\leq_B\ell_\infty$. In particular, if every $E$-equivalence class
is countable, then $E\leq_B \ell_\infty$.

\item[(5)] An equivalence relation is called {\it countable}
if every of its equivalence classes is countable. Among all
countable Borel equivalence relations there exists a maximum one in
terms of Borel reducibility \cite{DJK}. We denote any such
equivalence relation by $E_\infty$. By the remark above,
$E_\infty\leq \ell_\infty$. An equivalence relation $E$ is {\it
essentially countable} if $E\leq_B E_\infty$.

\item[(6)] The equivalence relation $=^+$ codes the equality
relation for countable sets of real numbers. In symbols, for
$\vec{x}=(x_n), \vec{y}=(y_n)\in \R^{\omega}$,
$$\vec{x}=^+ \vec{y}\iff \{x_n\,:\, n\in\omega\}=\{y_n\,:\, n\in\omega\}.$$
It is an easy consequence of a classical theorem of Luzin and
Novikov (see Theorem 18.10, \cite{Ke}) in descriptive set theory
that $E_\infty\leq_B\ =^+$. It is also known that
$\ell_{\infty}\nleq_B\ =^+$ (by results of Kechris and Louveau
\cite{KL}) and $=^+\,\nleq_B \ell_\infty$ (see below).

\item[(7)] The equivalence relation $E_0^\omega$ is defined on
$(2^\omega)^\omega$ as follows: for $\vec{x}=(x_n), \vec{y}=(y_n)\in
(2^\omega)^\omega$,
$$ \vec{x}E_0^\omega\vec{y}\iff \forall n\in\omega\ x_nE_0y_n. $$
$E_0^\omega$ has been studied explicitly or implicitly in, for example, \cite{So},
\cite{HK}, and \cite{HjKL}. Note that it is a ${\bf\Pi}^0_3$
equivalence relation. Results of Solecki \cite{So} imply that
$E_0^\omega$ is not essentially countable, that is,
$E_0^\omega\nleq_B E_\infty$. Further results of Hjorth, Kechris,
and Louveau \cite{HK} \cite{HjKL} imply that $E_0^\omega$ is not
Borel reducible to any ${\bf\Sigma}^0_3$ equivalence relations. Thus
in particular $E_0^\omega\nleq_B\ell_\infty$. It is a somewhat
elusive task to trace the references for this result; for the
convenience of the reader we will give a direct proof of it later in
this section.

The importance of $E_0^\omega$ lies in both the fact that it is
combinatorially easy to analyze and the speculation that it seems to
be the simplest (or even minimum) equivalence relation not reducible
to $\ell_\infty$. For instance, it is relatively simple to show that
$E_0^\omega\leq_B\ =^+$ (we will give a proof later in this
section); therefore it follows immediately that $=^+\,\nleq_B
\ell_\infty$. Also, when we need to consider equivalence relations
which seem to be more complex than $\ell_\infty$, the reducibility
of $E_0^\omega$ to them gives good test questions.

\item[(8)] The equivalence relation $E_G^\infty$ is the universal
orbit equivalence relation induced by Borel actions of Polish
groups. We shall not deal with general orbit equivalence relations
in this paper. Therefore we will omit the details of the definition
of $E_G^\infty$. The interested reader can find more information in
\cite{BK}, \cite{GK}, or \cite{Me}.

\item[(9)] Among all analytic equivalence relations on Polish spaces
there is a universal one, that is, every other analytic equivalence
relation is Borel reducible to it. Following \cite{LR} we denote it
by $E_{{\bf\Sigma}_1^1}$. As we mentioned in the Introduction this
equivalence relation plays an important role when natural
equivalence relations between separable Banach spaces are
considered. However, results in this paper only involve equivalence
relations much less complex than $E_{{\bf\Sigma}^1_1}$.
\end{itemize}

In the rest of the paper we will consider more equivalence
relations, but most of them will be closely related to
$\ell_\infty$.

In the remainder of this section we give some proofs of facts
related to $E_0^\omega$ (see (7) above) for the convenience of the
reader. We fix some notation to be used in these proofs, as well as
in the rest of the paper. First, we fix once and for all a
computable bijection
$\langle\cdot,\cdot\rangle\,:\,\omega\times\omega\,\to\,\omega$.
Next, we let $\omega^{<\omega}$ denote the countable set of all
finite sequences of natural numbers. For $s\in \omega^{<\omega}$ let
$|s|$ denote the length of $s$, that is, if $s=(s_1,\dots, s_n)$
then $|s|=n$. The empty sequence is denoted $\emptyset$, and we set
$|\emptyset|=0$. If $s=(s_1,\dots, s_n), t=(t_1,\dots, t_m)\in
\omega^{<\omega}$, then we let
$$ s*t=\left\{\begin{array}{ll} (s_1,\dots, s_n, t_{n+1},
\dots, t_m), & \mbox{ if $m>n$,} \\ s, & \mbox{ if $m\leq
n$.}\end{array}\right. $$ Then $|s*t|=\max\{|s|,|t|\}$, and $s*t$ is
obtained by replacing the first $|s|$ many elements of $t$ by $s$.
This definition also makes sense when $t\in\omega^{<\omega}$ is
replaced by an element of $\omega^\omega$.
For $s, t$ as above we also let
$$ s\oplus t=(s_1, t_1, s_2, t_2, \dots). $$
Then $|s\oplus t|=|s|+|t|$. This definition makes sense when both
$s$ and $t$ are replaced by elements of $\omega^\omega$. For $x,
y\in \omega^\omega$, $x\oplus y$ is obtained from shuffling the
elements of $x$ and $y$ into a single sequence.

\begin{lem}\label{lem:E0w=+}
$E_0^\omega \leq_B\ =^+$.
\end{lem}

\begin{proof} Let $s^0, s^1, \dots$ be an enumeration
of $\omega^{<\omega}$. Fix some $\vec{z}=(z_n)\in (2^\omega)^\omega$
such that for all $i\neq j\in\omega$, $(z_i,z_j)\not\in E_0$. For
$\vec{x}=(x_n)\in (2^\omega)^\omega$, let $f(x)=(y_n)$, where for
$n=\langle i,j \rangle$,
$$ y_n= z_i\oplus (s^j*x_i). $$
It is easy to verify that $f$ is a Borel (in fact continuous)
reduction from $E_0^\omega$ to $=^+$.
\end{proof}

\begin{lem} \label{Ew}
$E_0^\omega$ is not Borel reducible to any $\bs^0_3$ equivalence
relation.
\end{lem}

\begin{proof}
Suppose $X$ is a Polish space, $E$ a $\bs^0_3$ equivalence relation
on $X$, and $f\colon (2^\omega)^\omega \to X$ a Borel function such
that for all $\vec{x},\vec{y} \in (2^\omega)^\omega$,
$$ \vec{x}\, E_0^\omega \, \vec{y} \iff f(x) \, E \, f(y).$$
Since $f$ is Borel, and hence Baire measurable,
there is a comeager set $C \subseteq (2^\omega)^\omega$ such that
$f \res C$ is continuous.
We may assume $C$ is a $G_\delta$ set. We may now compute
$E_0^\omega \cap( C \times C)$ to be $\bs^0_3$, namely,
$$(x,y) \in E_0^\omega \cap (C \times C) \iff x,y \in C
\ \wedge\  (f(x),f(y)) \in E.$$
Since $f \res C$ is continuous and $E \in \bs^0_3$,
this shows $E_0^\omega \cap (C \times C)$ to be $\bs^0_3$.
To get a contradiction
it thus suffices to prove the following claim.

\begin{claim}
For every comeager set $C \subseteq (2^\omega)^\omega$,
$E_0^\omega \cap (C \times C)$ is $\bp^0_3$-complete.
\end{claim}
\begin{proof}
Let $B=\{ x \in 2^{\omega}\, \colon\,\forall i\in\omega\ \exists
j\in\omega\ \forall k \geq j \ x(\langle i,k \rangle) =1\,\}.$ Then
$B$ is clearly $\bp^0_3$. We first show that $B$ is
$\bp^0_3$-complete. For this, let $A \subseteq 2^\omega$ be
$\bp^0_3$, say $A=\bigcap_i \bigcup_j \bigcap_l A_{i,j,k}$, where
each $A_{i,j,k}$ is clopen.
 Define $\rho \colon 2^\omega \to 2^\omega$
as follows. For each $i,k \in \omega$ let $a_{x,i,k} \in \omega$
be the least integer $j \leq k$, if one exists, such that
$x \in A_{i,j,k'}$ for all $k' \leq k$. Let $\rho(x)(i,k)=1$
iff $a_{x,i,k}$ and $a_{x,i,k-1}$ are both defined and are equal.
Otherwise set $\rho(x)(i,k)=0$. The map $\rho$ is continuous
from $2^\omega$ to $2^{\omega}$,
and $x \in A$ iff $\rho(x) \in B$. Thus, $B$ is $\bp^0_3$-complete.

Note that $(2^\omega)^\omega$ is homeomorphic to
$2^{\omega\times\omega}$. For notational simplicity we work with
$2^{\omega\times\omega}$ below, and identify it with
$(2^\omega)^\omega$. If $s\in 2^{n\times m}$ for some
$n,m\in\omega$, then the basic clopen set determined by $s$, denoted
by $N_s$, is the set $\{x \in 2^{\omega\times\omega}\,:\, \forall
i<n, j<m\ x(i,j)=s(i,j)\}$. Write $C= \bigcap_n D_n$ where each
$D_n$ is open dense in $2^{\omega\times\omega}$.

We next define two continuous functions $\pi_1,\pi_2 \colon
2^{\omega} \to 2^{\omega \times \omega}$ so that $$x\in B \iff
(\pi_1(x),\pi_2(x))\in E_0^\omega \res (C \times C).$$ For each
sequence $s \in 2^n$ we will define  values $\pi_1(s), \pi_2(s) \in
2^{p(n) \times p(n)}$ for some $p=p(n)$ which depends only on $n$.
We will then take, for $x \in 2^\omega$, $\pi_1(x)=\bigcup_{n}
\pi_1(x\!\res\!n)$ and likewise for $\pi_2(x)$.

Suppose inductively that for some $n \in \omega$ and every sequence
$s \in 2^n$ we have defined $\pi_1(s), \pi_2(s) \in 2^{p \times p}$
for some $p=p(n) \in \omega$ which depends only on $n$. Suppose also
that $N_{\pi_1(s)}, N_{\pi_2(s)} \subseteq D_n$ for each $s \in
2^n$. Let $n+1=\langle i,k \rangle$. For each $s' \in 2^{n+1}$
extending $s \in 2^n$, extend $t_1:=\pi_1(s)$ and $t_2:=\pi_2(s)$ to
$t'_1$, $t'_2$ by letting $t'_1(i,p(n)+k)=t'_2(i,p(n)+k)=1$ if
$s(n)=1$, and otherwise setting $t'_1(i,p(n)+k)=0$,
$t'_2(i,p(n)+k)=1$. Extend $t'_1, t'_2$ to $t''_1$, $t''_2$ in
$2^{q_n \times q_n}$, where $q_n=p(n)+n$,
 by setting all other undefined values to $0$. Note that all of the
$t''_1$, $t''_2$  are elements of $2^{q_n \times q_n}$. Let $p(n+1)$
be large enough so that there is a finite function $h_{n+1} \colon
(p(n+1) \times p(n+1)) -(q_n \times q_n) \to \{ 0,1\}$ such that for
all of the $t''_1$, $t''_2$  we have that $u_1=t''_1 \cup h_{n+1}$
and  $u_2=t''_2 \cup h_{n+1}$ determine  basic open sets with $N_u
\subseteq D_{n+1}$. We can achieve this in $2 \cdot 2^{n+1}$ steps,
using the fact that $D_{n+1}$ is dense open. Set $\pi_1(s')=u_1$,
$\pi_2(s')=u_2$. Note that for any $s_1, s_2 \in 2^{n+1}$ and
$a,b\in\{1,2\}$, $\pi_a(s_1)$, $\pi_b(s_2)$ differ in at most one
point of  $(p(n+1) \times p(n+1)) -(p(n) \times p(n))$.

Clearly $\pi_1$, $\pi_2$ are continuous and $\pi_1(x)$, $\pi_2(x)
\in C$ for any $x \in 2^\omega$. If $x \in B$, then for each $i$ let
$k(i)$ be such that $x(\langle i,k\rangle)=1$ for all $k \geq k(i)$.
Fix $i \in \omega$. For any $n \geq \langle i, k(i) \rangle$, if
$n=\langle i,j\rangle$ for some $j$ then $\pi_1(x \!\res\! n)$,
$\pi_2(x \!\res\! n)$ are extended identically in going to $\pi_1(x
\!\res\! n+1)$ and $\pi_2(x \!\res\! n+1)$ (namely, they have value
$1$ at $(i, p(n)+j)$ and $0$ at the other new points of the domain).
If $n=\langle i',j\rangle$ where $i'\neq i$, then we still have that
$\pi_1(x \!\res\! n)$, $\pi_2(x \!\res\! n)$ are extended
identically on point of the form $(i,k)$ (they both are $0$ there).
So, $\pi_1(x)$, $\pi_2(x)$ agree at coordinates of the form $(i,k)$
for all large enough $k$. So, $(\pi_1(x)\,\pi_2(x))\not\in
E_0^\omega$.

Conversely, if $x \notin B$, then for some $i$, there are infinitely
many $j$ that $x(\langle i,j\rangle)=0$. Fix such an $i$. For each $j$
with $x(\langle i,j\rangle)=0$ let $n=\langle i,k \rangle$, and we
have that $\pi_1(x)(n)$ and $\pi_2(x)(n)$ disagree at $(i,
p(n)+j)$. This implies that $\neg \pi_1(x)\, E_0^\omega\, \pi_2(x)$.
\end{proof}
This completes the proof of lemma~\ref{Ew}.
\end{proof}

\section{\label{sec:PreB}Codings of separable Banach spaces and the local equivalence}

To apply the descriptive set theoretic framework to the study of
equivalence relations on separable Banach spaces, the collection of
separable Banach spaces must be viewed as a Polish space.

One way to do this is to use the well-known theorem of Banach and
Mazur that $C[0,1]$ is a universal separable Banach space, that is,
every separable Banach space is linearly isometric to a (necessarily
closed) subspace of $C[0,1]$. The collection of all separable Banach
spaces is then viewed as a subspace of the hyperspace $F(C[0,1])$
with the Wijsman topology (see Section~\ref{sec:Pre}). Let
$$ {\mathfrak B}=\{ F\in F(C[0,1])\,:\, \mbox{$F$ is a
linear subspace of $C[0,1]$}\,\}.$$ We check below that ${\mathfrak
B}$ is a Polish subspace.

\begin{lem} ${\mathfrak B}$ is a $G_\delta$ subspace of $F(C[0,1])$,
hence is Polish.
\end{lem}

\begin{proof}
Fix a countable dense $D\subseteq C[0,1]$. Let $d$ be the metric on
$C[0,1]$ given by its norm. We claim that for any $F\in F(C[0,1])$,
$F\in {\mathfrak B}$ iff
$$ \forall p, q, a, b\in {\mathbb Q}\ \forall x, y\in D\ [\, d(x,F)<a
\wedge d(y, F)<b \Longrightarrow d(px+qy, F)<|pa|+|qb|\, ]. $$ If $F$
is a linear subspace of $C[0,1]$ the demonstrated condition clearly
holds. Conversely, suppose the condition holds. Since $F$ is closed,
it suffices to show that for all $u,v\in F$ and $p,q\in{\mathbb Q}$,
$pu+qv\in F$. For this take two sequences $x_n, y_n\in D$ such that
$d(x_n,u), d(y_n,v)<2^{-n}$. Then $d(px_n+qy_n, F)<(|p|+|q|)2^{-n}$ by
the assumption, and $d(px_n+qy_n, pu+qv)<(|p|+|q|)2^{-n}$. Thus
$d(pu+qv, F)<(|p|+|q|)2^{-n+1}$.  Since $n$ is arbitrary, we have that
$d(pu+qv, F)=0$ and $pu+qv\in F$. The claim implies immediately that
${\mathfrak B}$ is $G_\delta$ in the Wijsman topology of $F(C[0,1])$.
\end{proof}

In discussing Banach spaces with a distinguished Schauder basis
another representation is often used. A fundamental result of Pe\l
czy\'nski \cite{Pe} says that there is a universal basis $U=(e_i)$,
that is, for every separable Banach space with a basis $(X,B)$,
where $B=(x_i)$, there is a one-to-one map $f\colon B \to U$ which
extends to a linear isomorphism from $X$ to the space spanned by
$f(B)$. In this manner, the collection of separable Banach spaces with
a basis can be identified with the Cantor space $2^\omega$. For $x
\in 2^\omega$, let $X_x$ be the separable Banach space with a basis
coded by $x$. The space of all Banach spaces with infinite bases correspond to
$[\omega]^\omega$, the set of all infinite subsets of $2^\omega$, which is a
$G_\delta$ subset of $2^\omega$ and a Polish space in its own right.

In practice it is often easier to work with the following direct coding for Banach spaces with infinite bases.
Fix once and for all for the rest of the paper
an enumeration $(s^n)$ of $\Q^{<\omega}$, the set of all finite
sequences of rational numbers.  To any separable Banach space with a
infinite basis $(Y, (y_i))$, we associate a sequence of real numbers
$(r_n)\in \R^\omega$, where
$$ r_n=\big\|\sum_{i=1}^k a^n_i y_i\big\|_Y $$ if $s^n=(a^n_1,\dots,
a^n_k)$. Recall that a normalized (i.e., all $y_i$ have norm $1$)
basis $(y_i)$ is called {\em monotone} if the projections onto
initial segments of $(y_i)$ have norm $1$. There is no loss of
generality in restricting to monotone bases, since for every
normalized basis we can take an equivalent norm for which the basis
is monotone.
Let ${\mathfrak B}_b\subseteq \R^\omega$ be the set of all possible
sequences $(r_n)$ associated with Banach spaces with a monotone
basis.

Again we check below that ${\mathfrak B}_b$ is a Polish
space. We henceforth use the phrase ``Banach space with basis''
to denote a pair $(X, B)$, where $X$ is a Banach space, and $B$ is a
basis. It is these objects that are coded by the reals in $\mathfrak{B}_b$.

\begin{lem}\label{lem:Bb} ${\mathfrak B}_b$ is a closed  subspace of $\R^\omega$,
hence is Polish.
\end{lem}

\begin{proof} For notational convenience in this proof we identify
$s^n=(a^n_1,\dots, a^n_k)$ with the infinite sequence $(a^n_1,\dots,
a^n_k, 0, 0, \dots)$. Then it makes sense to speak of
$s^n+s^m\in\Q^{<\omega}$ for any $n, m\in\omega$, and $ps^n\in \Q$
for any $p\in\Q$ and $n\in\omega$.  Now for any $(r_n)\in\R^\omega$,
$(r_n)\in {\mathfrak B}_b$ iff all of the following hold:
\begin{enumerate}
\item[(i)] if $s^n=(0,\dots, 0,1,0.\dots,0)$, then $r_n=1$;
\item[(ii)] if $s^m$ coincides with an initial segment of $s^n$,
then $r_m \leq r_n$.
\item[(iii)] if $s^n+s^m=s^l$, then $r_l\leq r_n+r_m$;
\item[(iv)] for any $p\in \Q$, if $s^m=ps^n$, then $r_m=|p|r_n$.
\end{enumerate}
The conditions listed are closed for $(r_n)$ in $\R^\omega$. Note
that (i) and (ii) imply that the basis is monotone which implies
that any non-zero linear combination of the $y_i$ has positive norm.
\end{proof}

Given any $(r_n)\in {\mathfrak B}_b$, by the proof of
Lemma~\ref{lem:Bb} we can associate a Banach space with an infinite basis
whose norm function is approximated by the sequence $(r_n)$. In this
manner each element of ${\mathfrak B}_b$ codes a Banach space with
basis. For $y \in {\mathfrak B}_b$, let $Y_y$ be the space coded by
$y$.

We remark that the two codings for Banach spaces with bases are
equivalent in the following precise sense.  It is easy to see that
there is a continuous function $\varphi \colon [\omega]^\omega \to
{\mathfrak B}_b$ such that for all $x \in [\omega]^\omega$, $X_x$ is
linearly isometric to $Y_{\varphi(x)}$. Conversely, by the proof of Pe\l
czy\'nski's theorem \cite{Pe} there is also a Borel function $\psi
\colon {\mathfrak B}_b \to [\omega]^\omega$ such that for all $y \in
{\mathfrak B}_b$, $Y_y$ is linearly isomorphic to $X_{\psi(y)}$.

As for the relationship between codings using elements of ${\mathfrak B}$ versus those of ${\mathfrak B}_b$, we denote by ${\mathfrak B}_\beta$ the subspace of ${\mathfrak B}$ consisting of all linear subspaces of $C[0,1]$ admitting bases. It follows immediately from the proof of
the Banach-Mazur theorem that there is a Borel function $\Phi: {\mathfrak B}_b\to {\mathfrak B}_\beta\subseteq {\mathfrak B}$ such that for all $y\in {\mathfrak B}_b$, $Y_y$ is linearly isometric
to $\Phi(y)$. Intuitively, in defining $\Phi(y)$ one omits the given basis and obtains an isomorphic (in fact isometric) copy of the space as
a subspace of $C[0,1]$. It is easy to see that ${\mathfrak B}_\beta$ coincides with the isomorphic saturation of $\Phi({\mathfrak B}_b)$, denoted $[\Phi({\mathfrak B}_b)]$,
which is also
the same as the isometric saturation of $\Phi({\mathfrak B}_b)$. Obviously both $\Phi({\mathfrak B}_b)$ and ${\mathfrak B}_\beta=[\Phi({\mathfrak B}_b)]$ are analytic subsets of
${\mathfrak B}$. However, it is not known whether either of them is Borel.

Rosendal has pointed out that the function $\Phi$ can be improved to be injective, that is, there is a Borel {\it injective} function $\Psi$ with all the above properties.
To see this, fix $\lambda : {\mathfrak B}_b\to [0,1]$ a Borel injection and $\varphi: C[0,1]\oplus_\infty C[0,1]\to C[0,1]$ a linear isometric embedding. For any $y\in {\mathfrak B}_b$, let
$$\Psi(y)=\left\{\, \varphi(v,\lambda(y)v)\in C[0,1]\,:\, v\in \Phi(y)\subseteq C[0,1]\,\right\}. $$
Then $\Psi(y)$ and $\Phi(y)$ are linearly isometric, and $\Psi$ is obviously injective
because of the injectivity of $\lambda$ and $\varphi$. It follows that $\Psi({\mathfrak B}_b)$ is Borel. Note that ${\mathfrak B}_\beta=[\Psi({\mathfrak B}_b)]=[\Phi({\mathfrak B}_b)]$ and the question about its
Borelness remains unresolved.

Next we turn to equivalence relations between separable Banach spaces.

We remark first that the uniform homeomorphism relation is analytic as
an equivalence relation on either ${\mathfrak B}$ or ${\mathfrak
B}_b$.  This was noted in \cite{FLR}, and in fact it it proved there
that the uniform homeomorphism relation on all Polish metric spaces is
complete analytic.  For the convenience of the reader we recall the
following argument. Let $\approx$ denote the uniform homeomorphism
relation on ${\mathfrak B}$.  Then for $X, Y\in {\mathfrak B}$,
$X\approx Y$ iff there exist $(x_n), (y_n)\in C[0,1]^\omega$ such that
\begin{enumerate}
\item[(a)] $x_n\in X$ and $y_n\in Y$ for all $n\in\omega$;
\item[(b)] the sets $D_X:=\{x_n\,:\, n\in\omega\}$ and
$D_Y:=\{y_n\,:\, n\in\omega\}$ are dense in $X$ and $Y$ respectively;
\item[(c)] the map $f: D_X\to D_Y$ with $f(x_n)=y_n$ for all
$n\in\omega$ is a uniformly continuous bijection, with $f^{-1}$ also
uniformly continuous.
\end{enumerate}
One direction of the equivalence is clear. For the other direction,
we note that the uniform homeomorphism $f$ defined on a dense set
$D_X$ can be uniquely extended to a necessarily uniform
homeomorphism of the entire space, since Cauchy sequences in $D_X$
will correspond to Cauchy sequences in $D_Y$ by the uniform
continuity of $f$ and $f^{-1}$.  Now the conditions (a) through (c)
are all Borel conditions for $X, Y, (x_n)$, and $(y_n)$. Hence
$\approx$ is analytic. It also follows immediately that the uniform
homeomorphism relation on ${\mathfrak B}_b$ is analytic via the
pullback of the Borel function $\Phi$ defined above.

In the remainder of this section we define a notion of local
equivalence inspired by Ribe's theorem \cite{Ri1} and study its
basic properties.  In doing this we recall some concepts and results
from Banach space theory. All other unexplained terms and facts can
be found in \cite{BL} or \cite{T}.

Recall that, for linearly isomorphic Banach spaces $X$ and $Y$, the
\emph{Banach-Mazur distance} between $X$ and $Y$ is defined as
$$d(X,Y):=\inf\{\,\|T\| \|T^{-1}\|\,:\, T:X\to Y\ \mbox{is an
isomorphism}\,\}.$$ The following theorem is a fundamental result
about uniform homeomorphism.

\begin{thm}[Ribe \cite{Ri1}]\label{thm:Ribe}
If $X$ and $Y$ are uniformly homeomorphic Banach spaces, then there
exists a constant $C>0$ such that for every finite-dimensional
subspace $E$ of $X$ there exists a finite-dimensional subspace $F$
of $Y$ such that $d(E, F)\le C$, and vice versa.
\end{thm}

This motivates the following concept.

\begin{defn}
Let $X$ and $Y$ be Banach spaces. We say that $X$ and $Y$ are {\it
locally equivalent}, and denote by $X\loce Y$, if there exists a
constant $C>0$ such that for every finite-dimensional subspace $E$
of $X$ there exists a finite-dimensional subspace $F$ of $Y$ such
that $d(E, F)\le C$, and vice versa.
\end{defn}

Here we refer to the structure of finite-dimensional subspaces of a
Banach space as its local structure. In the literature the local
equivalence between $X$ and $Y$ is sometimes informally referred to
as $X$ and $Y$ having the same finite-dimensional subspaces. Ribe's
theorem states that uniformly homeomorphic spaces are locally
equivalent. The converse is not true. For instance, as we mentioned
in the Introduction, $\ell_p$ is not uniformly homeomorphic to $L_p$
for $1\le p<\infty$, $p\neq 2$; however, they are locally
equivalent.

In the following we compute the descriptive complexity of the local
equivalence as an equivalence relation on either the Polish space
${\mathfrak B}$ of all separable Banach spaces or the Polish space
${\mathfrak B}_b$ of all separable Banach spaces with basis.

\begin{lem}
Local equivalence is a $\bs^0_3$ equivalence relation on either
${\mathfrak B}$ or ${\mathfrak B}_b$.
\end{lem}

\begin{proof}
First consider $\loce$ as an equivalence relation on ${\mathfrak
B}_b$. Let $(X, (e_i))$ and $(Y,(f_i))$ be the Banach spaces with basis
coded as elements of $\mathfrak{B}_b$ by $x,y \in \mathbb{R}^\omega$.
Note that every finite-dimensional subspace of $(X, (e_i))$ can be
approximated by a space with a (finite) basis consisting of finite
rational linear combinations of the $e_i$. We use the enumeration
$\{s^n\}$ of $\Q^{<\omega}$ in the definition of ${\mathfrak
B}_b$. For $s^n=(a^n_1, \dots, a^n_k) \in \Q^{< \omega}$, let
$s^n(X)=\sum_{i=1}^k a^n_i e_i$. For $\vec n=n_1,\dots, n_N\in\omega$, let
$X_{\vec n}$ be the $\leq N$-dimensional subspace of $X$ with
basis $s^{n_1}(X), \dots, s^{n_N}(X)$. Similarly we define $s^m(Y)$
and $Y_{\vec m}$ for $\vec m= m_1,\dots, m_M\in\omega$.

Let $I$ be the set of $\vec n =(n_1,\dots, n_N) \in \N^{<\omega}$  such that the vectors
$s^{n_1}, \dots, s^{n_N}$ are linearly independent. Note that if $X \in \mathfrak{B}_b$
is coded by
$x$,  and $\vec n \in I$, then the vectors  $s^{n_1}(X),\dots, s^{n_N}(X)$
are linearly independent in the space $X$ (since $x$ codes a monotone
basis for $X$).

With this
notation we have that $X\loce Y$ iff
$$\exists M\geq 1\ \ \forall N\geq 1\ \forall (n_1, \dots, n_N) \in I \ \exists
(m_1, \dots, m_N) \in I\ d(X_{n_1,\dots, n_N}, Y_{m_1,\dots, m_N})<M $$ and
vice versa. It suffices  to show that for fixed $\vec n$ and $\vec m$ that the relation on
$\mathfrak{B}_b$ given by $$U(x,y) \Leftrightarrow
d(X^x_{n_1,\dots, n_N}, Y^y_{m_1,\dots, m_N}) < M$$ is open,
where $X^x$ and $Y^y$ denote the spaces coded by $x$ and $y$.
Fix $x,y$ with $U(x,y)$. Let $T \colon X^x_{\vec n} \to Y^y_{\vec m}$
be a linear isomorphism with $\|T\| \|T^{-1}\| < M$. Let $x_1,\dots, x_N$ denote
$s^{n_1}(X),\dots, s^{n_N}(X)$, and let $y_1=T(x_1),\dots, y_N=T(x_N)$.
For $x' \in \mathfrak{B}_b$, let $x'_1,\dots, x'_N$ denote $s^{n_1}(X'),\dots,
s^{n_N}(X')$ where $X'$ is the space coded by $x'$. Let $T' \colon X'_{\vec n} \to Y$
be the linear map defined by $T'(x'_1)=y_1, \dots, T'(x'_N)=y_N$.

It suffices by symmetry to show that for any
$\epsilon >0$ there is an open $V \subseteq \R^\omega$ containing $x$
such that for all $x' \in \mathfrak{B}_b \cap V$ we have
$| \| T\| -\| T'\| |<\epsilon$. Let $\rho>0$ be such that
$\rho < \min \{ \| x_i \| \} \leq \max \{ \| x_i \| \}< \frac{1}{\rho}$. Let
$0< \eta  < \inf \{ \| \alpha_1 x_1 + \cdots + \alpha_N x_N \|_X
\colon \vec \alpha \in S_N \}$, where $S_N=\{ \vec \alpha \colon \sum \alpha_i^2=1\}$.
 By definition of the product topology on
$\R^\omega$, there is clearly an open set $V_1$ about $x$ such that for $x' \in V_1$ we have
$\rho < \min \{ \| x'_i \| \} \leq \max \{ \| x'_i \| \}< \frac{1}{\rho}$.
It thus suffices to show that for all $\epsilon >0$
there is a neighborhood $V \subseteq V_1$ of $x$ such that for all $x' \in V$ we have that
$| \| \alpha_1 x_1 + \cdots + \alpha_N x_N \|_X -
\| \alpha_1 x'_1 + \cdots + \alpha_N x'_N \|_{X'}  | <
\epsilon \left( \frac{\eta^2 \rho}{2 N \| T\| } \right)$, for all
$\vec \alpha\in S_N$. For then, letting $v= \alpha_1 x_1 + \cdots + \alpha_N x_N $
and $v'=\alpha_1 x'_1 + \cdots + \alpha_N x'_N $ we have (noting $T(v)=T(v')$ and assuming
$\epsilon < \frac{\eta}{2}$):

\begin{equation*}
|\frac{ \| T(v)\|}{\| v\|} -\frac{\| T(v')\|}{\| v'\| } |=
\frac{\| T(v)\| }{\| v\| \| v'\|}
(| \| v\| -\| v'\||) \leq
\frac{\| T\| N} {\rho}  \frac{2}{\eta^3} (| \| v\| -\| v'\||)
\leq \epsilon
\end{equation*}

Let $\mathfrak{N} \subseteq S_N \cap \Q^N$ be such that
for all $\vec \alpha \in S_N$, there is a $\vec q \in \mathfrak{N}$
such that $|\alpha_i - q_i| < \frac{\epsilon \rho \delta}{3N}$ for all $i$,
where $\delta=\frac{\eta^2 \rho}{2 N \| T\| }$.
(the rational points on $S_N$ are dense).
Given $\vec \alpha \in S_N$, let $\vec q \in \mathfrak{N}$ be such that
$|\alpha_i-q_i|< \frac{\epsilon \rho \delta}{3N}$ for all $1 \leq i \leq N$.
We have that $| \| \sum \alpha_i x_i\| -\| \sum q_i x_i \| |<
(\frac{\epsilon \rho \delta}{3N})(N) \max \{ \| x_i\|\}
\leq \frac{\epsilon \delta}{3}$, with a similar estimate for the $x'_i$. Since the $q_i$ and
$s^{n_i}$ are rational, if $V$ is a small enough neighborhood of $x$ and
$x' \in V$ we will have $| \| \sum q_i x_i \| - \| \sum q_i x'_i\| |< \frac{\epsilon \delta}{3}$.
Thus, $| \| \sum \alpha_i x_i\| -\| \sum \alpha_i x'_i\| |<\epsilon \delta$.

Next consider $\loce$ as an equivalence relation on ${\mathfrak B}$.
Fix a countable dense $D\subseteq C[0,1]$. For this part of the
proof let $d$ be the metric on $C[0,1]$ given by the norm. Let $Q$
be the set of all quadruples $(s, t, n, q)$ such that
\begin{enumerate}
\item[(1)] $s, t\in D^{<\omega}$ (the set of all finite
sequences of elements of $D$), $n\in\omega$, $q\in\Q$;
\item[(2)] there is some $k\in\omega$ such that $s, t\in D^k$,
that is, $|s|=|t|=k$;
\item[(3)] if $s=(s_1,\dots,s_k)$, $t=(t_1,\dots,t_k)$,
then for any $x_1,\dots, x_k, y_1,\dots,y_k\in C[0,1]$ such that
$d(s_i,x_i), d(t_i,y_i)<2^{-n}$ for $1\leq i\leq k$, letting $T$ be
the linear map from $\spn(x_1, \dots, x_k)$ to $\spn(y_1,\dots,y_k)$
with $T(x_i)=y_i$ for $1\leq i\leq k$, we have that
$\|T\|\|T^{-1}\|< q$.
\end{enumerate}

Note that if $x_1,\dots, x_k, y_1,\dots, y_k\in C[0,1]$ and the
linear map $T:\spn(x_1, \dots, x_k)\to \spn(y_1,\dots,y_k)$ sending
$x_i$ to $y_i$ satisfies $\|T\|\|T^{-1}\|<C$ for some $C>0$, then
there is a quadruple $(s,t,n,q)\in Q$ such that $q<C$ and for all
$1\leq i\leq k$, $d(s_i,x_i), d(t_i,y_i)<2^{-n}$.

We claim that for any $X, Y\in F(C[0,1])$, $X\loce Y$ iff
$$\begin{array}{rl}
&\exists C \in \Q\ \forall k\in\omega\ \forall z_1,\dots,
z_k\in D\ \forall \epsilon\in\Q\ \\
&\hspace{20pt}\{\ \forall 1\leq i\leq k \
 d(z_i, X)<\epsilon \Longrightarrow \exists (s, t, n, q)\in Q\
  [\,q < C\, \wedge\, 2^{-n}<\epsilon\, \wedge \\
& \hspace{40pt}\forall 1\leq i \leq k\ (\,d(s_i, z_i)<\epsilon\,
\wedge\, d(t_i, Y)<2^{-n}\, \wedge\, d(s_i, X)<2^{-n}\,)\,]\ \}.
\end{array}
$$

To prove the claim, first suppose $X\loce Y$, and let $C>0$ be a
witness. Suppose $z_1,\dots, z_k \in D$ and $\epsilon$ are given,
and for $1\leq i \leq k$ let $x_i\in X$ be such that $d(x_i,
z_i)<\epsilon$. By the local equivalence between $X$ and $Y$ there
are $y_1, \dots, y_k\in Y$ such that the linear map
$T:\spn(x_1,\dots, x_k)\to\spn(y_1,\dots, y_k)$ sending $x_i$ to
$y_i$ satisfies that $\|T\|\|T^{-1}\|<C$. We have a quadruple
$(s,t,n,q)\in Q$ such that $q<C$, $d(s_i, x_i), d(t_i, y_i)<2^{-n}$
for all $1\leq i\leq k$. Moreover, we may choose $n$ to be large
enough such that $2^{-n}<\epsilon$ and $d(s_i, z_i)<\epsilon$. This
verifies the displayed property.

For the other implication, let $C$ be as in the displayed property.
Let $x_1, \dots, x_k\in X$ be given. Let $\epsilon>0$ be
sufficiently small compared with $k$.  Let $z_1,\dots, z_k \in D$
with $d(z_i,x_i)<\epsilon$ for all $1\leq i\leq k$. Then the
displayed property gives a quadruple $(s,t,n,q)\in Q$. Thus $q<C$,
$2^{-n}<\epsilon$, and for $1\leq i\leq k$, $d(s_i,z_i)<\epsilon$,
$d(t_i, Y)<2^{-n}$, $d(s_i,X)<2^{-n}$. In particular there are
$y_1,\dots,y_k \in Y$ such that $d(t_i, y_i)<2^{-n}$ for all $1\leq
i\leq k$, and by the definition of $Q$ the linear map
$T:\spn(s_1,\dots, s_k)\to \spn(y_1,\dots, y_k)$ sending $s_i$ to
$y_i$ satisfies that $\|T\|\|T^{-1}\|<q$. Since $d(s_i,
x_i)<2\epsilon$, and $\epsilon$ is sufficiently small, the map
$S:\spn(x_1,\dots, x_k)\to \spn(y_1,\dots, y_k)$ sending $x_i$ to
$y_i$ satisfies that $\|S\|\|S^{-1}\|<C+1$.

The displayed property is apparently $\bs^0_3$ in the Wijsman
topology on $F(C[0,1])$. It follows that $\loce$ is $\bs^0_3$ on
${\mathfrak B}$.
\end{proof}

We can also consider the local equivalence on the space $2^\omega$
of codes for Banach spaces with basis via the Pe\l czy\'nski
universal basis. Recall that there is a continuous function
$\varphi: [\omega]^\omega\to {\mathfrak B}_b$ such that for any $x\in
2^\omega$, $X_x$ is linearly isometric to $Y_{\varphi(x)}$.
Via this map the local equivalence on $[\omega]^\omega$ is continuously
reduced to $\loce$ on ${\mathfrak B}_b$. It follows that the local
equivalence on $2^\omega$ is also $\bs^0_3$.

Now it follows from Lemma~\ref{Ew} that $E_0^\omega$ is not Borel
reducible to $\loce$ on either ${\mathfrak B}$ or ${\mathfrak B}_b$,
and by Lemma~\ref{lem:E0w=+} $=^+$ is not Borel reducible also to
either of them. In Section~\ref{sec:Ribe} we will prove that in fact
$\loce$ (on either ${\mathfrak B}$ or ${\mathfrak B}_b$) is Borel
bireducible to $\ell_\infty$, thus completely determine its
complexity in the Borel reducibility hierarchy.

\section{\label{sec:cons}The uniform homeomorphism on a class of Banach spaces}

In this section we construct a class of Banach spaces and completely
characterize its uniform homeomorphism relation. In the construction
and proofs we use a few well-known results in Banach space theory.
Our standard reference for undefined terms and unexplained results
is \cite{T}.

Recall that two given bases $(x_i)$ and $(y_i)$ of Banach spaces are
said to be \emph{$C$-equivalent} for $C>0$ if there exist positive
constants $A, B$ with $AB\le C$ such that for all scalar sequences
$(a_i)$,
$$\frac{1}{A}\big\|\sum_i a_i x_i\big\|\le \big\|\sum_i a_i y_i\big\|\le
B\big\|\sum_i a_i x_i\big\|.$$

We will make use of the following important notion in the study of
the local structures of Banach spaces. Let $X$ be a Banach space and
let $1\le p\le 2$. The \emph{type $p$ constant $T_{p,n}(X)$ of $X$
over $n$ vectors} is the smallest positive number such that for
arbitrary $n$ vectors $x_1, \ldots, x_n\in X$,

\begin{equation}\label{type-def}\left(\underset{\ep_i=\pm
1}{\textrm{Ave}}\big\|\sum_{i=1}^n \ep_i x_i\big\|^2\right)^{1/2}\le
T_{p,n}(X) \big(\sum_{i=1}^n\|x_i\|^p\big)^{1/p}.
\end{equation}

$X$ is said to have \emph{type $p$} if $T_p(X)=\sup_n
T_{p,n}(X)<\infty$. For $\ell_q^n$ spaces type $p$ constants can be
easily computed and they satisfy the following estimates
\begin{eqnarray}
\label{type-estimates}n^{\max(0, 1/q - 1/p)}&\le& T_p(\ell_q^n)\le c
q^{1/2}n^{\max(0, 1/q - 1/p)}\ \textrm{for}\ 1\le q<\infty,
\end{eqnarray}
where $c$ is a universal constant. Moreover, the proof of the lower
estimate in (\ref{type-estimates}) also shows that for $n\le k$,
\begin{equation}\label{type-n-estimates}
T_{p,n}(\ell_q^k)\le c q^{1/2}n^{\max(0, 1/q - 1/p)}.
\end{equation}

Note that type is a linear notion, in particular, if $T:Y\to X$ is a
linear embedding then $T_{p,n}(Y)\le \|T\| \|T^{-1}\| T_{p,n}(X)$.

For a sequence $\vec{p}=(p_i)\in (1,2)^\omega$ by $S_{\vec{p}}$ we
will denote the $\ell_2$-direct sum of finite-dimensional
$\ell_{p_i}^{n_i}$ spaces for a fixed sequence of increasing
dimensions $(n_i)$. That is,

$$S_{\vec{p}}:=\big(\sum_{i=1}^{\infty}\oplus\,
\ell_{p_i}^{n_i}\big)_2.$$

The next theorem singles out a class of such spaces on which
isomorphism, local equivalence and uniform homeomorphism relations
all coincide.

\begin{thm} \label{cond}
Let $I_i=[l_i, r_i]$ be a sequence of successive intervals in
$(1,2)$. Then there exists $(n_i)\in \omega^\omega$ such that for
$\vec{p}=(p_i)$ and $\vec{q}=(q_i)$ with each $p_i, q_i\in I_i$ we
have that $S_{\vec{p}}$ is uniformly homeomorphic to $S_{\vec{q}}$
if and only if there exists a constant $C\ge 1$ such that
$$n_i^{|\frac{1}{p_i}-\frac{1}{q_i}|}\le C$$ for all $i\in\omega$.
\end{thm}

\begin{proof}
For any sequence of dimensions $(n_i)$ if
$n_i^{|\frac{1}{p_i}-\frac{1}{q_i}|}\le C$ for some $C<\infty$, then
$d(\ell_{p_i}^{n_i},\ell_{q_i}^{n_i})\le C$ for all $i\in\omega$. In
fact, in this case the unit vector bases are $C$-equivalent. From
this it easily follows that $S_{\vec{p}}$ and $S_{\vec{q}}$ are
$C$-isomorphic, and in particular, they are uniformly homeomorphic.

Let $I_i=[l_i, r_i]$ be a sequence of given intervals in $(1,2)$.
Pick a sequence $(n_i)$ of natural numbers such that

\begin{equation}\sup n_i^{\frac{1}{l_i}-\frac{1}{r_i}}=\infty  \mbox{ and } \label{n_i}n_{i+1}^{1/r_{i+1}}\ge n_i^{1/l_i}, \ \ i=1,2,3,\dots.
\end{equation}

Suppose without loss of generality that $p_i< q_i\in [l_i, r_i]$
with $\sup_i n_i^{1/p_i -1/q_i}=\infty$. By Ribe's theorem
\ref{thm:Ribe} it is sufficient to show that the spaces in the
sequence $(\ell_{p_i}^{n_i})$ do not linearly embed in $S_{\vec{q}}$
with a uniform constant.

Fix $i_0\in\omega$. Put $n=n_{i_0}$, $p=p_{i_0}$ and let
$T:\ell_p^n\to S_{\vec{q}}$ be a linear embedding. Since
$$T_{2,n}(\ell_p^n)\le \|T\| \|T^{-1}\| T_{2,n}(S_{\vec{q}}),$$
and $T_{2,n}(\ell_p^n)\ge n^{1/p-1/2}$, we need an upper estimate for
$T_{2,n}(S_{\vec{q}})$.

Let $I_{i_0}$ be such that $p, q=q_{i_0}\in I_{i_0}$. That is, $q_{i_0
-1}<p<q_{i_0}=q$. Let $x^1, \ldots, x^n\in S_{\vec{q}}$. Then,
writing each $x^j$ as $\sum_{i=1}^{\infty}x_i^j$ where $x_i^j\in
\ell_{q_i}^{n_i}$, we have

\begin{eqnarray*}
\underset{\ep_j=\pm 1}{\textrm{Ave}}\big\|\sum_{j=1}^n \ep_j
x^j\big\|_{S_{\vec{q}}}^2&=&\underset{\ep_j=\pm
1}{\textrm{Ave}}\big\|\sum_{j=1}^n \ep_j
\big(\sum_{i=1}^{\infty}x_i^j\big)\big\|_{S_{\vec{q}}}^2=\underset{\ep_j=\pm
1}{\textrm{Ave}}\sum_{i=1}^{\infty}\big\|\sum_{j=1}^n \ep_j
x_i^j\big\|_{q_i}^2\\
&\le& \sum_{i<i_0}\underset{\ep_j=\pm
1}{\textrm{Ave}}\big\|\sum_{j=1}^n \ep_j x_i^j\big\|_{q_i}^2 +
\sum_{i\ge i_0}\underset{\ep_j=\pm
1}{\textrm{Ave}}\big\|\sum_{j=1}^n \ep_j x_i^j\big\|_{q_i}^2\\
&\le& \sum_{i<i_0} T_2^2(\ell_{q_i}^{n_i})\sum_{j=1}^n
\|x_i^j\|_{q_i}^2 + \sum_{i\ge i_0}
T_{2,n}^2(\ell_{q_i}^{n_i})\sum_{j=1}^n \|x_i^j\|_{q_i}^2.
\end{eqnarray*}

Using the estimates (\ref{type-estimates}) for $i<i_0$ and
(\ref{type-n-estimates}) for $i\ge i_0$ sums, the last inequality is
less than or equal to

$$\sum_{i<i_0} c^2 q_i n_i^{2/q_i -1}\sum_{j=1}^n \|x_i^j\|_{q_i}^2 +
\sum_{i\ge i_0} c^2 q_i n^{2/q-1}\sum_{j=1}^n \|x_i^j\|_{q_i}^2,$$
which is, by  (\ref{n_i}), less than

$$ 2c^2 n^{2/q-1} \sum_{j=1}^n\sum_{i=1}^{\infty}\|x_i^j\|_{q_i}^2
=2c^2n^{2/q -1}\sum_{j=1}^n \|x^j\|_{S_{\vec{q}}}^2.
$$

Thus, we have shown that $T_{2,n}(S_{\vec{q}})\le \sqrt{2}c
n^{1/q-1/2}$. It follows that
$$\|T\| \|T^{-1}\|\ge \frac{n^{1/p-1/2}}{\sqrt{2}c n^{1/q-1/2}}
=\frac{n^{1/p-1/q}}{\sqrt{2}c}.$$
\end{proof}

\section{\label{sec:Ribe}The complexity of the uniform homeomorphism and the local equivalence}

In this section we prove the main theorems of our paper. In doing
this we also define some natural equivalence relations and
characterize their complexity. Some of the equivalence relations to be defined in this section have already been considered in
\cite{Ro}. For instance, Lemma~\ref{emba}, Definition~\ref{proder}, and the beginning of Theorem~\ref{thm:linfX} can be found in \cite{Ro}.
For the sake of completeness we give all definitions and proofs in a self-contained manner.

For notational clarity we use the following convention in this
section. Let $X$ be a set. We use $\vec{x}$ to denote an element of
$X^\omega$, the set of the all infinite sequences of elements of
$X$. The coordinates of $\vec{x}$ will be denoted by $x(n)$ for
$n\in\omega$. Thus $\vec{x}=(x(n))=(x(0), x(1), \dots)$. This is
slightly different from previous sections, but it provides the most
convenience for the arguments of this section.

Recall that the equivalence relation $E_{\ell_\infty}$ (simply
$\ell_\infty$ when there is no danger of confusion) is the
equivalence relation on $\R^\omega$ defined by
$$\vec x \, E_{\ell_\infty} \, \vec  y \iff \exists C\ \forall n\ |x(n)-y(n)| <C
$$
for $\vec{x},\vec{y}\in \R^\omega$. We consider the following
variation. Let $B$ be the set of all infinite increasing sequences
of positive real numbers without an upper bound. For any $\vec
b=(b(0),b(1), \dots)\in B$, we denote by $E_{\ell_\infty}^{\vec b}$
the equivalence relation $E_{\ell_\infty}$ restricted to the set
$\prod_{n\in\omega} [0,b(n)]$.

\begin{lem} \label{emba} For any $\vec{b}\in B$,
$E_{\ell_\infty} \leq_B E_{\ell_\infty}^{\vec b}$.
\end{lem}

\begin{proof}
For each $n\in\omega$ let $\rho_n$ be a linear map from
$[-b(n),b(n)]$ onto $[0,b(n)]$. Define $\pi \colon \R^\omega \to
\prod_{n\in\omega}[0,b(n)]$ by
$$
\pi(\vec{x})(\langle i,j\rangle)=\left\{\begin{array}{ll} \rho_j(x(i)),
& \mbox{if $x(i) \in [-b(j), b(j)]$,} \\
0, & \mbox{if $x(i)< -b(j)$,} \\
b(j), & \mbox{if $x(i)> b(j)$,}
\end{array}\right.
$$
for all $i,j\in\omega$. Clearly $\pi(\vec x) \in \prod_{n\in\omega}[0,b(n)]$.
Note that if $\pi(\vec x_1)=\vec y_1$, $\pi(\vec x_2)=\vec y_2$, then
$$|y_1(n)-y_2(n)| \leq |x_1(i)-x_2(i)|$$ for all $n=\langle i,j\rangle\in\omega$.
Thus $\vec x_1 E_{\ell_\infty} \, \vec x_2$ implies $\pi(\vec x_1)
E_{\ell_\infty}^{\vec b} \pi(\vec x_2)$. Suppose $\vec x_1$ is not
$E_{\ell_\infty}$-equivalent to $\vec x_2$. Then for any $C>0$,
there is an $i\in\omega$ such that $|x_1(i)-x_2(i)|>C$. Let $j$ be a
large enough integer such that $b(j)>\max \{ |x_1(i)|,|x_2(i)|\}$.
Let $n=\langle i,j\rangle$. Then $y_1(n)=\rho_j(x_1(i))$ and
$y_2(n)=\rho_j(x_2(i))$, and so $|y_1(n)-y_2(n)|>C/2$. So, $\pi(\vec
x_1)$ is not $E_{\ell_\infty}^{\vec b}$-equivalent to  $\pi(\vec
x_1)$. This shows that $\pi$ is a reduction from $E_{\ell_\infty}$
to $E_{\ell_\infty}^{\vec{b}}$. It is clear that $\pi$ is a Borel
function.
\end{proof}

We are now ready to prove our first main theorem.

\begin{thm} \label{embthm} The equivalence relation $\ell_\infty$
is Borel reducible to the uniform homeomorphism relation on either
${\mathfrak B}_b$ or ${\mathfrak B}$.
\end{thm}

\begin{proof} By Lemma~\ref{emba} it suffices to define a
Borel reduction from $E_{\ell_\infty}^{\vec{b}}$ (for some
$\vec{b}\in B$) to the uniform homeomorphism relation for Banach
spaces with a basis. To construct the Banach spaces we use the proof
of Theorem~\ref{cond}. For this fix $\vec{b}\in B \cap
\omega^\omega$ with $b(0)>0$. For all $i\in\omega$ put
$$ \delta_i=\displaystyle\frac{1}{b(i)2^i}\ \mbox{ and }\
n_i= 2^{\frac{1}{\delta_i}}=2^{b(i) 2^i}. $$ Let $I_i$ be a sequence
of successive intervals in $(1,2)$ with $|I_i|=2^{-i-1}$. Assume
$I_i=[l_i,r_i]$. Since $n_i^2 \leq n_{i+1}$, equation~(\ref{n_i}) is
satisfied. The sequences $(n_i)$ and $(I_i)$ will be used as in the
proof Theorem~\ref{cond}. Let $\sigma_i$ be the affine bijection
between $[0,b(i)]$ and $I_i$. For $\vec x \in
\prod_{n\in\omega}[0,b(n)]$, define $\rho(\vec x) \in \R^\omega$ by
$\rho(\vec x)(i)=\sigma_i(x(i))$. Finally, define $\pi(\vec x)=
S_{\rho(\vec x)}$. So, for all $\vec x \in
\prod_{n\in\omega}[0,b(n)]$, $\pi(\vec x)$ is a separable Banach
space with a basis.

We show that $\pi$ is the desired reduction. It is straightforward
to check that $\pi$ is Borel as a map from
$\prod_{n\in\omega}[0,b(n)]$ to ${\mathfrak B}_b$. Granting that
$\pi$ is a reduction, then composed with the Borel map $\Phi:
{\mathfrak B}_b\to {\mathfrak B}$, it would be a reduction to the
uniform homeomorphism relation on ${\mathfrak B}$.

To verify that $\pi$ is a reduction, consider $\vec x_1, \vec
x_2\in\prod_{n\in\omega}[0,b(n)]$. From Theorem~\ref{cond} we have
that $S_{\rho(\vec x_1)}$ and  $S_{\rho(\vec x_2)}$ are uniformly
homeomorphic iff
$$\exists C>0\ \forall i\in\omega\ \ n_i^{\left|\frac{1}{\rho(\vec x_1)(i)}-
\frac{1}{\rho(\vec x_1)(i)}\right|} <C.$$ By taking logarithm we get
that this is equivalent to
$$ \exists D>0\ \forall i\in\omega\ \ \left|\frac{\log(n_i)}{\rho(\vec x_1)(i)}
-\frac{\log(n_i)}{\rho(\vec x_2)(i)}\right|<D.$$ Using the
definition of $\rho$ we get that the inequality is equivalent to
$$\log(n_i) |\sigma_i(\vec x_1(i))- \sigma_i(\vec x_2(i))| \frac{1}{\sigma_i(\vec x_1(i)) \cdot
\sigma_i(\vec x_2(i))} <D.$$
Since $\sigma_i(\vec x_1(i)) \in (1,2)$ and likewise for
$\vec x_2$, the statement is thus equivalent to
$$\exists D>0\ \forall i\in\omega\ \ \log(n_i) |\sigma_i(\vec x_1)(i)-
\sigma_i(\vec x_2)(i)| <D.$$ By the linearity of $\sigma_i$, we in
fact have
$$|\sigma_i(\vec x_1(i))- \sigma_i(\vec x_2(i))|=
\frac{|x_1(i)-x_2(i)|}{b(i) \cdot 2^{i+1}}.$$ Finally, our choice of
$(n_i)$ guarantees that
$$\frac{1}{2} \leq \left|\frac{\log(n_i)}{b(i) \cdot 2^i}\right| \leq 2. $$
Therefore, the statement is eventually equivalent to $\exists D>0\
\forall i\in\omega\ |\vec x_1(i) -\vec x_2(i)| <D$, that is, $\vec
x_1 E_{\ell_\infty}^{\vec b} \vec x_2$.
\end{proof}

Theorem~\ref{thm:main1} is now a direct corollary of the above
proof. In particular we have the following corollary.

\begin{thm} \label{thm:lbloce} The equivalence relation $\ell_\infty$ is
Borel reducible to the local equivalence on either ${\mathfrak B}_b$
or ${\mathfrak B}$, that is, $\ell_\infty\leq_B\ \loce$.
\end{thm}

This gives a half of Theorem~\ref{thm:main2}. Next we prove
Theorem~\ref{thm:main2} by showing the reverse reduction. We will
use the following concept and lemma.

\begin{defn}
For $X=(X,d)$ a Polish metric space, let $F_X$ be the equivalence
relation on $X^\omega$ defined by
$$
\vec x \, F_X \, \vec y \iff \exists C>0 \ [\,\forall i\ \exists j\ d(x(i),y(j))<C
\wedge \forall i\ \exists j\ d(y(i),x(j)) <C\,].
$$
\end{defn}

\begin{lem} \label{fx}
For every Polish metric space $(X,d)$,
$F_X \leq_B \ell_\infty$.
\end{lem}

\begin{proof}
Fix a $1$-net $R=\{ r_0, r_1, \dots\}$ in $X$. We define $\pi:
X^\omega \to \R^\omega$ by
$$ \pi(\vec{x})(i)=d(r_i, \{ x(0), x(1), \dots\}). $$
It is easy to check that $\pi$ is a Borel function. We verify that
it is a reduction from $F_X$ to $E_{\ell_{\infty}}$. Suppose $\vec x
\, F_X \, \vec y$, and let $C>0$ be a witness. For any $z \in X$, if
$\delta =d(z, \{ x(0), x(1), \dots\})$, then $d(z, \{ y(0),
y(1),\dots\}) \leq \delta +C$. So, $|d(z, \{ x(0), x(1),\dots\})
-d(z,\{ y(0), y(1),\dots\}) | \leq C$. Thus, $\pi(\vec x) \,
E_{\ell_\infty} \, \pi(\vec y)$.

Conversely, suppose $\vec x$ is not $F_X$-equivalent to $\vec y$.
Let $C>0$ be arbitrary. Then there is a $k$ such that $d(x(k), \{
y(0), y(1),\dots\})
>C$ or $d(y(k), \{ x(0), x(1),\dots\})>C$.
Without loss of generality, assume the former.
Let $i$ be such that $d(x(k),r_i)<1$. Then
$\pi(\vec x)(i)<1$, but $\pi(\vec y)(i)> C-1$.
So $\pi(\vec x)$ is not $E_{\ell_\infty}$-equivalent to $\pi(\vec y)$.
\end{proof}

\begin{thm} \label{erthm} The local equivalence on either
${\mathfrak B}$ or ${\mathfrak B}_b$ is Borel reducible to the
equivalence relation $\ell_\infty$, that is, $\loce\, \leq_B
\ell_\infty$.
\end{thm}

\begin{proof}
Let $\sF$ be the collection of finite dimensional Banach spaces
(presented with bases). The following distance function is a
separable metric on $\sF$. If two spaces $(X,(x_1,\dots, x_n))$ and
$(Y,(y_1,\dots, y_n))$ are both $n$-dimensional, let
$\rho_n(X,Y)=\max \{ \log (\norm T\norm), \log(\norm T^{-1}
\norm)\}$, where $T:X\to Y$ is the linear isomorphism sending $x_i$
to $y_i$. By truncating  we may assume each $\rho_n \leq 1$, and we
may then put together the $\rho_n$ to obtain a metric on $\sF$ (if
$\dim(X) \neq \dim(Y)$ we set $\rho(X,Y)=1$). Let $(\bar{\sF},\rho)$
be the Polish space obtained by completing $\rho$.

First consider $\loce$ as an equivalence relation on ${\mathfrak
B}_b$. Given a separable Banach space $X$ with basis $(x_i)$, we
define $\pi(X) \in \bar{\sF}^\omega$ as follows. Let
$\chi_1,\chi_2,\dots$ enumerate $(\Q^{< \omega})^{< \omega}$.
Suppose $\chi_i=(\vec q_1, \dots, \vec q_k)$. Then let $\pi(X)(i)$
be the $k$-dimensional subspace of $X$ with basis  $(e_1, \dots,
e_k)$, where $e_i=\sum \vec q_i(j) x_j$ for all $1\leq i\leq k$. It
is routine to check that $\pi$ is a Borel function. Note that every
finite-dimensional subspace of $X$ is approximated arbitrarily
closely in the $\rho$ metric by a term of the sequence $\pi(X)$. It
is then straightforward from the definition of the local equivalence
that $X \loce  Y$ iff $\pi(X)\ F_{\bar \sF}\ \pi(Y)$. Thus, $\pi$ is
a reduction from $\loce$ to $E_{\bar \sF}$. We are done by
Lemma~\ref{fx}.

We modify the above argument to work for $\loce$ as an equivalence
relation on ${\mathfrak B}$. Let $d$ be the metric on $C[0,1]$ given
by the norm. Let $D \subseteq C[0,1]$ be countable dense. Fix an
enumeration of $D^{<\omega} \times \omega$ as $(s_0,n_0), (s_1,n_1),
\dots$. Fix a Borel function $\sigma \colon F(C[0,1]) \to C[0,1])$
such that $\sigma(F) \in F$ for all nonempty $F \in F(C[0,1])$ (see
Theorem 12.13, \cite{Ke}). For $x \in C[0,1]$, $F \in F(C[0,1])$,
and $n \in \omega$, let
$$\sigma_n(x,F)=\left\{\begin{array}{ll}
\sigma(F \cap \{u\in C[0,1]\,:\, d(x,u)\leq \frac{1}{n}\}),
& \mbox{if $F \cap \{u\in C[0,1]\,:\, d(x,u)\leq \frac{1}{n}\}\neq\emptyset$,} \\
\sigma(F), & \mbox{otherwise.}
\end{array}\right.
$$
Given a separable Banach space $X\in {\mathfrak B}$, let
$\pi(X)(i)\in \bar\sF$ code the finite-dimensional subspace of $X$
with basis $(\sigma_{n_i}(s_i(0), F), \dots,
\sigma_{n_i}(s_i(|s_i|-1),F))$. Since $D$ is dense, every
finite-dimensional subspace of $X$ is approximated arbitrarily
closely by spaces of the form $\pi(F)(i)$.  Thus, $\pi$ is a Borel
reduction from $\loce$ to $E_{\bar \sF}$.
\end{proof}

Theorem~\ref{thm:main2} is immediate from Theorems~\ref{thm:lbloce}
and \ref{erthm}. To summarize, we have shown that the local
equivalence between separable Banach spaces has the same complexity
as $\ell_\infty$, and the uniform homeomorphism relation is at least
as complex as $\ell_\infty$. Thus we have obtained the sharpest
result possible for the uniform classification by considering the
local structures of Banach spaces alone.

The equivalence relation $F_X$ we used in the above proof is sort of
a generalization of the $\ell_\infty$ equivalence relation on the
space of countable subsets of a Polish metric space equipped with
the Hausdorff metric. In the remainder of this section we consider a
full generalization of $\ell_\infty$ to arbitrary Polish metric
spaces and characterize its complexity.

\begin{defn}\label{proder} Let $X=(X,d)$ be a Polish metric space.
The equivalence relation $E_{\ell_\infty(X)}$, or simply
$\ell_\infty(X)$, on $X^\omega$ is defined as
$$ \vec{x}E_{\ell_\infty(X)}\vec{y}\iff \exists C>0\ \forall i\in\omega\
 d(x(i),y(i))<C. $$
\end{defn}

We have the following dichotomy for the complexity of
$\ell_\infty(X)$ in the Borel reducibility hierarchy for any Polish
metric space $X$.

\begin{thm} \label{thm:linfX}
Let $X=(X,d)$ be a Polish metric space with $d$ unbounded.
Then $\ell_\infty(X)$ is Borel bireducible with either $\ell_\infty$ or $E_1$.
\end{thm}

\begin{proof}
We first reduce $\ell_\infty(X)$ to $\ell_\infty$. Fix a countable
{\it $1$-net} $R=\{ r_0, r_1, \dots\}$ in $X$, that is, $R \subseteq
X$ with $d(r_i,r_j) > 1$ for all $i \neq j$, and $\forall x \in X\
\exists i\in\omega\  d(x,r_i) \leq 1$. This can be done in any
separable metric space. Define $\pi: X^\omega\to \R^\omega$ by
$$\pi(\vec x)(\langle i,j\rangle)=d(x(i),r_j) $$
for any $i,j\in\omega$. Then $\pi$ is continuous, in particular
Borel. Let $\vec{x}, \vec{y}\in X^\omega$. If $\vec x \,
E_{\ell_\infty(X)} \, \vec y$, then let $C>0$ be such that
$d(x(i),y(i)) <C$ for all $i\in\omega$. Then it follows that for any
$j\in\omega$, $|d(x(i),r_j)-d(y(i),r_j)| <C$. So $\pi(\vec x) \,
E_{\ell_\infty} \, \pi(\vec y)$. Conversely, if $\pi(\vec x) \,
E_{\ell_\infty} \, \pi(\vec y)$, then for some $C>0$ we have that
$\forall i\ \forall j\ |d(x(i),r_j)-d(y(i),r_j)| <C$. It we take $j$
so that $d(x(i), r_j) \leq 1$, then this implies that
$d(x(i),y(i))<C+2$, and so $\vec x \, E_{\ell_\infty(X)} \, \vec y$.


Next we reduce $E_1$ to $\ell_\infty(X)$. Since $\R$ is Borel
isomorphic to both $2^\omega$ and $\omega^\omega$, we may work with $E_1$ defined on
either $(2^\omega)^\omega$ or $(\omega^\omega)^\omega$, whichever is more convenient. Fix a sequence $(z_n)\in X^\omega$ with $\lim_n
d(z_0,z_n)=\infty$. Define $\tau: (2^\omega)^\omega\to X^\omega$ by
$$ \tau(\vec{x})(\langle i,j\rangle)=\left\{ \begin{array}{ll}
z_i, & \mbox{if $x_i(j)=1$,} \\
z_0, & \mbox{otherwise,}
\end{array}\right.
$$
for all $i,j\in\omega$. Again $\tau$ is continuous, hence Borel.
Given $\vec x, \vec x'\in (2^\omega)^\omega$ and $n\in\omega$, we
have that $\forall k \geq n \ x_k=x'_k$ iff $\tau(\vec{x})$,
$\tau(\vec{x}')$ only disagree where they take values in $\{z_0,
\dots, z_{n-1}\}$. This implies that $\vec x E_1 \vec x'$ iff
$\tau(\vec{x})E_{\ell_\infty(X)}\tau(\vec{x}')$.

We have shown so far that $E_1 \leq_B \ell_\infty(X) \leq_B
\ell_\infty$. If $Y$ is a $1$-net in $X$, then clearly
$\ell_\infty(X)$ is bireducible with $\ell_\infty(Y)$. So, without
loss of generality we may assume that $X$ is countable. For every
positive $C$, let $\sim_C$ be the equivalence relation on $X$ which
is the transitive closure of the relation $\{ (x,y) \colon
d(x,y)<C\}$. We call the $\sim_C$-equivalence classes the {\it
$C$-components} of $X$. We consider now two cases.

\underline{Case I}: For all $C$ there is a bound $K_C$ on the
diameter of the $C$-components.

In this case we reduce $\ell_\infty(X)$ to $E_1$. For each positive
integer $n$, let $A^n_0, A^n_1, \dots$ enumerate (with repetition)
the $n$-components of $X$. Given $\vec x=(x_0, x_1,\dots) \in
X^\omega$, define $\vec y=\pi(\vec x) \in (\omega^\omega)^\omega$ by:
$y_n(m)=j$ iff $x_m \in A^n_j$. Suppose first that $\vec x \,
E_{\ell_\infty(X)} \, \vec x'$, say $\forall n\ d(x_n, x'_n) \leq
N$. Then, for all $n \geq N$ we have that for all $m$, $x_m$ and
$x'_m$ lie in the same $n$-component, since any two points in two
distinct $n$-components have $d$ distance greater than $n$. This
shows that $y_n=y'_n$ for all $n \geq N$, and so $\vec y \, E_1 \,
\vec y'$. Conversely, suppose $\forall n \geq N\ y_n=y'_n$. So, for
all $n \geq N$ and all $m$, $x_m$ and $x'_m$ lie in the same
$n$-component. In particular, $x_m$ and $x'_m$ lie in the same
$N$-component for all $m$, and so $d(x_m,x'_m) \leq K_m$ for all
$m$, that is, $\vec x\, E_{\ell_\infty(X)}\, \vec x'$.

\underline{Case II}: For some $C$  and every $K$, there is a
$C$-component of diameter greater than $K$.

In this case we reduce $\ell_\infty$ to $\ell_\infty(X)$. Fix $C$ as
in the case hypothesis. By a { \em $C$-path } we mean a finite
sequence of points $y_0, y_1, \dots, y_n$ from $X$ such that
$d(y_i,y_{i+1}) <C$ for all $i$. Note that all the points of a
$C$-path lie in the same $C$-component of $X$. Let $p_0, p_1, \dots$
enumerate all of the $C$-paths in $X$. If $p=(y_0,\dots,y_n)$ is a
$C$-path and $i \in \omega$, let $p(i)=y_i$ if $i \leq n$ and
otherwise let $p(i)=y_n$. Clearly $d(p(i),p(j)) \leq C |i-j|$ for
any $C$-path $p$ and any $i,j \in \omega$. Given $x \in
\omega^\omega$,  define $\vec y=\pi(x) \in X^\omega$ by $y_{\langle
i,j\rangle}= p_i(x(j))$. If $\forall m\ |x(m)-x'(m)| \leq N$, then
$\forall m\ d(y_m,y'_m) \leq C N$ from the above observation.
Suppose then that $x$ is not $E_{\ell_\infty}$-equivalent to $x'$.
Given $k$, let $A \subseteq X$ be a $C$-component with diameter
greater than $k$. Let $z,w \in A$ with $d(z,w)>k$. Let $p$ be a
$C$-path from $z$ to $w$. Say $p=(z=z_0,z_1, \dots, w=z_n)$. Let $i$
be such that $|x(i)-x'(i)| >n$. From $p$ we can easily obtain a
$C$-path $q$ such that $q(x(i))=z_0$ and $q(x'(i))=z_n$ (have the
path $q$ start at $z_0$, remain at $z_0$ for an appropriate number
of steps, then follow $p$, and then remain at $z_n$). Say $q=q_j$.
Then $y_{\langle j, i \rangle}=q(x(i))=z_0$ and $y'_{\langle j, i
\rangle}=q(x'(i))=z_n$. Thus $d(y_{\langle j, i \rangle},
y'_{\langle j, i \rangle})\geq k$. Since this is true for all $k$,
we have that $\vec y$ is not $E_{\ell_\infty(X)}$-equivalent to
${\vec{y}}'$.
\end{proof}

\section{\label{sec:class}Some special classes of separable Banach spaces}

In this section we generalize the construction in
Section~\ref{sec:cons} to obtain some classes of separable Banach
spaces. For each of these classes it turns out that the isomorphism,
the uniform homeomorphism, and the local equivalence relations on it
coincide. We also obtain some characterizations for the possible
complexity of these equivalence relations.

We will use the following equivalence relation on $2^\omega$ and a
characterization of its possible complexity.

\begin{defn}
For any sequence $\vec t=(t_i)\in\R^\omega$ with $t_i\geq 0$ for all
$i\in\omega$, let $E_{\vec t}$ be the equivalence relation on
$2^\omega$ defined by
$$x\, E_{\vec{t}}\ y\iff \sup_i\,( t_i \cdot | x(i)-y(i)| )< \infty.$$
\end{defn}

\begin{thm} \label{ge0e1}
For any $\vec t\in \R^\omega$ with $t_i\geq 0$ for all $i\in\omega$,
$E_{\vec t}$ is either smooth, Borel bireducible with $E_0$, or
Borel bireducible with $E_1$.
\end{thm}

\begin{proof}
If $(t_i)$ is bounded, then $E_{\vec t}$ is trivial, and in
particular smooth. So we assume $\vec t$ is unbounded. We
inductively define a finite or infinite sequence $n_0 < n_1 < \dots$
of natural numbers as follows. Let $n_0$ be the least $n\in\omega$,
if one exists, such that $\{ i \colon t_i \leq n\}$ is infinite.
Suppose $n_k$ is defined, then let $n_{k+1}$ be the least $n>n_k$,
if one exists, such that $\{ i \colon n_k < t_i \leq n \}$ is
infinite. If $n_k$ is defined, we also let $A_k=\{ i \colon n_{k-1}
< t_i \leq n_k\}$.

First assume that $n_k$ is defined for all $k\in\omega$. Thus, $A_k$
is defined for all $k$ and the $A_k$ form a partition of $\omega$.
Note that each $A_k$ is infinite by definition. Let $e_k^i$,
$i\in\omega$, enumerate $A_k$. Define $f: 2^\omega\to
(2^\omega)^\omega$ by $f(x)_k(i)=x(e_k^i)$. Clearly $x_1 \, E_{\vec
t} \ x_2$ iff the sequences of reals coded by $f(x_1)$ and $f(x_2)$
are eventually the same, that is, $f(x_1) E_1 f(x_2)$. Thus, $f$ is
a Borel reduction of $E_{\vec t}$ to $E_1$. In fact, $f$ is a
bijection between $2^\omega$ and $(2^\omega)^\omega$, so its inverse
gives a reduction from $E_1$ to $E_{\vec t}$.

Suppose next that $n_0$ is not defined. In this case $t_i \to
\infty$. Then in fact $x \, E_{\vec t} \ y$ iff $x \, E_0 \, y$,
that is, the identity map is a reduction from $E_{\vec t}$ to $E_0$.
Since the identity map is again a bijection, we have that $E_{\vec
t}$ is Borel bireducible with $E_0$.

Finally, suppose that $n_0 < \cdots < n_\ell$ are defined, while
$n_{\ell+1}$ is not. Since $(t_i)$ is unbounded, we must have that
$\omega -\bigcup_{k \leq \ell} A_k$ is infinite. Let $e_k^i$,
$i\in\omega$, enumerate $A_k$ for $k \leq \ell$, and let
$e^i_{\ell+1}$, $i\in\omega$, enumerate $\omega -\bigcup_{k \leq
\ell} A_k$. Define $g: 2^\omega\to 2^\omega$ by
$g(x)(i)=x(e^i_{\ell+1})$. Clearly $f$ is a Borel reduction of
$E_{\vec t}$ to $E_0$. For the other direction, define $h:
2^\omega\to 2^\omega$ by
$$ h(y)(j)=\left\{\begin{array}{ll}
y(i),, & \mbox{if $j=e^i_{\ell+1}$,} \\
0, & \mbox{otherwise.}
\end{array}\right.
$$
Easily $h$ is a reduction of $E_0$ to $E_{\vec t}$.
\end{proof}

The above proof can be simplified in view of known facts about $E_1$ and $E_0$ (see Section 2). In fact, if $E\leq_B E_1$ then $E$ is either smooth or Borel bireducible with either $E_0$ or $E_1$ by the dichotomy theorems of \cite{HKL} and \cite{KL}. However, we gave the full proof here since it is self-contained and gives some information about the combinatorial structure of the equivalence relation $E_{\vec{t}}$. This will happen again for the proof of Theorem~\ref{bthm} below.

As in Section~\ref{sec:cons} we consider sequences $\vec p= (p_i)$,
$\vec q = (q_i)\in\R^\omega$, and $\vec n=(n_i)\in\omega^\omega$
such that
\begin{equation}\label{eq:class1}
1 < p_i < q_i < p_{i+1} <2, \ n_i>0 \mbox{ and }
n_{i+1}^{\frac{1}{q_{i+1}}} > n_i^{\frac{1}{p_i}}.
\end{equation}
Let $\pqc$ be the collection of Banach spaces of the form
$$X= \big( \sum_{i=0}^\infty \oplus\, \ell_{r_i}^{n_i} \big)_2,$$
where $r_i \in \{ p_i, q_i \}$. $\pqc$ can be viewed as  a closed
subspace of ${\mathfrak B}_b$. To see this, first code the elements
of $\pqc$ by elements of $2^\omega$ in the natural manner (i.e.,
$x(i)$ determines whether to use $\ell_{p_i}$ or $\ell_{q_i}$). By
using a fixed bijection between $\omega \times \omega$ and $\omega$,
we fix an order of enumeration of the basis elements for all the
spaces in $\pqc$. This induces a map $f$ from $2^\omega$ to
$\mathfrak{B}_b \subseteq \R^\omega$ which is easily seen to be
continuous. Then $f(2^\omega)$ is a closed subset of
$\mathfrak{B}_b$ which represents the set of spaces in $\pqc$.
Clearly each $\pqc$ contains continuum many elements.

\begin{thm} \label{e0e1}
For any $\vec p$, $\vec q$, $\vec n$ satisfying
\mbox{\rm(\ref{eq:class1})} above, the uniform homeomorphism
relation on $\pqc$ is either smooth, Borel bireducible to $E_0$, or
Borel bireducible to $E_1$.
\end{thm}

\begin{proof}
Consider the sequence of numbers $t_i=n_i^{\frac{1}{p_i}-\frac{1}{q_i}}$.
First suppose that the sequence $(t_i)$ is bounded. In this case,
all of the spaces $\pqc$ are isomorphic, so the uniform homeomorphism
relation on $\pqc$ is trivial.

Suppose next that $(t_i)$ is unbounded. For each
$X \in \pqc$, let $z(X) \in 2^\omega$ be the real $z$ such that
$z(i)=0$ if $X$ involves $\ell_{p_i}^{n_i}$ and $z(i)=1$
if $X$ involves $\ell_{q_i}^{n_i}$.

From the proof of Theorem~\ref{cond} we have that for $X, Y \in \pqc$,
$X$ is uniformly homeomorphic to $Y$ iff
$\sup_i (t_i \cdot |z(X)(i)- z(Y)(i)|) < \infty$, that is, $z(X)\ E_{\vec{t}}\, z(Y)$.
Therefore we are done by Theorem~\ref{ge0e1}.
\end{proof}

We now extend Theorem~\ref{e0e1} to some even larger classes of
separable Banach spaces. Again we define and study some new
equivalence relations.

\begin{defn}
For any sequence $\vec B=(B_i)$ where each $B_i$ is a finite subset
of $\R$, let $E_{\vec B}$ denote the equivalence relation
$\ell_\infty$ restricted on $\prod_{i\in \omega} B_i$.
\end{defn}

For $\vec{B}=(B_i)$ as in the above definition let $b_i=\sup\{
|a|\,:\, a\in B_i\}$. Then $E_{\vec{B}}$ is also
$E^{\vec{b}}_{\ell_\infty}$ restricted to $\prod_{i\in\omega}B_i$.
Thus $E_{\vec{B}}\leq_B E_{\ell_\infty}^{\vec{b}}\leq_B
\ell_\infty$.

\begin{thm} \label{bthm}
Let $\vec B=(B_i)$ where each $B_i$ is a finite subset of $\R$.
Then $E_{\vec B}$ is either smooth, Borel bireducible with $E_0$,
Borel bireducible with $E_1$, or Borel bireducible with $\ell_\infty$.
\end{thm}

\begin{proof}
By translating each $B_i$ we may assume that each $B_i$ consists of
nonnegative real numbers and contains $0$ as its least element. Let
$b_i= \max B_i$. If $(b_i)$ is bounded, then $E$ is a trivial
equivalence relation, and so is smooth. So, we assume $(b_i)$ is
unbounded. Also, we may assume that $B_i \subseteq \omega$, for we
may replace $B_i$ by $\{ \lfloor a \rfloor \colon a \in B_i\}$.

For $i$, $n \in \omega$, let $F^i_n$ denote the finite equivalence
relation on $B_i$ given by the transitive closure of the relation
$$x R^i_n \, y \iff |x-y| \leq n.$$
For each $i$, $n$, let $a^i_n(0), \dots, a^i_n(k)$
enumerate the $F^i_n$ classes of $B_i$ in increasing order
(i.e., $\max (a^i_n(l)) < \min(a^i_n(l+1))$). Here $k=k(i,n)$
depends on $i$ and $n$.

First consider the case where for some $n$, there is no bound
on the size of the $F^i_n$ equivalence classes.
That is, $\forall b\ \exists i\ \exists l \ |a^i_n(l)|>b$.
Fix such an $n$. Let $i_0<i_1< \dots$ be a subsequence
and $l_0, l_1, \dots$ a sequence such that $|a^{i_m}_n(l_m)| > m$.
We know that $E_{\ell_\infty} \leq E_{\vec C}$
where $C_i=\{ 0, 1, \dots, i\}$. So, it suffices to show in this case that
$E_{\vec C} \leq E_{\vec B}$, as it then gives that $E_{\vec B}$
is Borel bireducible with $\ell_\infty$. Let $Z=\prod C_i$ and $X=\prod B_i$.
Define $\pi \colon Z \to X$ by
$$\pi(z)(i)=\left\{\begin{array}{ll}
0, & \mbox{if $i \notin \{ i_0, i_1,\dots\}$,} \\
\mbox{the $z(m)$-th element of $a^{i}_n(l_m)$}, & \mbox{if $i=i_m$.}
\end{array}\right.
$$
Then for all $x,y \in Z$ we have $|x(m)-y(m)| \leq
|\pi(x)(i_m)-\pi(y)(i_m)| \leq n |x(m)-y(m)|$. It follows that $\pi$
is a Borel reduction from $E_{\vec C}$ to $E_{\vec B}$.

Next consider the case where for each $n$ there is a bound $K_n$ on
the size of the $F^i_n$ equivalence classes, that is, $\forall i\
\forall l\ |a^i_n(l)|< K_n$. We first show in this case that
$E_{\vec B} \leq_B E_1$. We define a map $\tau$ from $X=\prod B_i$
to $(\ww)^\omega$ as follows. For $x \in X$ let $\tau(x)(n) \in
\ww=y_n$ be the real such that $y_n(i)=$ the unique $l$ such that
$x(i)\in a^i_n(l)$. Consider $x,y \in X$. If $x E_{\vec B}\, y$,
then for some $C>0$ we have $|x(i)-y(i)|<C$ for all $i$. Let $n$ be
such that $n>C$. Then for all $i$ we must have that $x(i)$, $y(i)$
lie in the same class of $F^i_n$, since any two points in distinct
$F^i_n$ class are at least $n$ apart. This shows that
$\tau(x)(m)=\tau(y)(m)$ for all $m \geq n$. That is, $\tau(x)
E_1(\omega)\, \tau(y)$, where $E_1(\omega)$ refers to the $E_1$
(eventual agreement) relation on $(\ww)^\omega$. Conversely, suppose
$\tau(x) E_1(\omega)\, \tau(y)$. Fix $n$ so that for all $m \geq n$,
$\tau(x)(m)=\tau(y)(m)$. Then for all $i$ we have that $|x(i)-y(i)|
\leq n K_n$, since any two points in the same $F^i_n$ equivalence
class are at most $n K_n$ apart. Thus, $x E_{\vec B}\, y$. Thus,
$\tau$ is a Borel reduction of $E_{\vec B}$ to $E_1(\omega)$.
However, it is easy to see that $E_1(\omega)\leq_B E_1$, so $E_{\vec
B} \leq_B E_1$ in this case.

We next consider subcases. First assume that for every $n$ and every
$M$ there is a $D_n$ such that for infinitely many $i$ we have that
$M< g^i_n <D_n$, where $g^i_n$ is the minimum distance between
distinct $F^i_n$ classes (and $=0$ if there is only one $F^i_n$
class). We may therefore get a sequence $k_0< k_1<\cdots$ such that
for all $n$, there are infinitely many $i$ such that $k_n< g^i_{n}
\leq k_{n+1}$. For each $n$, let $A_n=\{ i \colon k_n <  g^i_n <
k_{n+1} \}$. Note that the $A_n$ are pairwise disjoint (this follows
from the fact that $g^i_m \leq  g^i_n$ if $m>n$). for each $n$ and
each $i \in A_n$, let $l=l(i,n)$ and $l'=l'(i,n)$ be such that the
distance between $a^i_n(l)$ and $a^i_n(l')$ is between $k_n$ and
$k_{n+1}$. Define $\varphi \colon (2^\omega)^\omega \to X$ as
follows. Given $y=(y_0, y_1, \dots) \in (2^\omega)^\omega$, let
$\varphi(y)=x \in X$ where $x(i)=0$ if $i \notin \bigcup_n A_n$, and
for $i \in A_n$, say if $i$ is the $j^{\text{th}}$ element of $A_n$,
then $x(i)$ is the least element of $a^i_n(l)$ if $y_n(j)=0$ and the
least element of $a^i_n(l')$ if $y_n(j)=1$. Note that for any $x,
x'$ in the range of $\varphi$, we always have that for all $i \in
A_n$ that $|x(i)-x'(i)| \leq k_{n+1}+n K_n$. Also, if $y_n \neq
y'_n$, then for some $i \in A_n$ we have that $|x(i)-x'(i)| \geq
k_n$. It follows that $\varphi$ is a Borel reduction of $E_1$ to
$E_{\vec B}$. Thus, $E_{\vec B}$ is Borel bireducible with $E_1$.

Finally assume that for some $n$ we have that
$$\forall C>0\ \exists i_C\ \forall i \geq i_C\ g^i_n >C.$$
Define $\psi \colon
X \to \ww$ by $\psi(x)(i)=$ the unique $l$ such that
$x(i) \in a^i_n(l)$. If $x E_{\vec B} \, y$, then we must have $\psi(x)
E_0\, \psi(y)$ as $g^i_n$ tends to infinity with $i$. Conversely, if $
\psi(x) E_0\, \psi(y)$, then $x E_{\vec B} \, y$ since
$|x(i)-y(i)| \leq n K_n$ for all $i$. So, $E_{\vec B} \leq E_0$
in this case. More generally, if we assume that for some $n$
and some $M$ that
$$\forall C>0\ \exists i_C\ \forall i \geq i_C\ (g^i_n >C \vee
g^i_n<M),$$
then the same conclusion follows. This is because on the set $A$
of $i$ such that $g^i_n <M$ we have that $B_i$ consists of a single
$F^i_M$ class, and thus $|B_i|\leq K_M$ for these $i$. Thus,
$\max B_i\leq M K_M$ for such $i$, and so we proceed
as before to define $\psi$, except now we use
only those $i \notin A$ (i.e., set $\psi(x)(i)=0$ for $i \notin A$).
It is also easy to reduce $E_0$ to $E_{\vec B}$ in this case and
so $E_{\vec B}$ is Borel bireducible with $E_0$.
The argument is an easier variation of that given in
the preceding paragraph.
\end{proof}

To define our generalized classes of Banach spaces, we again fix a
sequence of successive intervals $I_i=[l_i,r_i]$ with $l_{i+1}>r_i$,
and integers $\vec n=(n_i)$ as in Theorem~\ref{cond}. Once again, we
assume that
$$n_{i+1}^{\frac{1}{r_{i+1}}} \geq n_i^{\frac{1}{l_i}}. $$
For each $i$, let $S_i \subseteq [l_i,r_i]$ be a finite set.

\begin{defn}
For $I_i$, $n_i$, and $S_i$ as above, let $\snc$
be the collection of separable Banach spaces of the form
$X= \left( \sum_{i=1}^\infty \oplus\, \ell_{r_i}^{n_i} \right)_2$,
where $r_i \in S_i$. Let $E_{\vec S, \vec n}$ denote the
uniform homeomorphism relation on the collection
$\snc$.
\end{defn}

We note that $\snc$ can be regarded as a closed subspace of
${\mathfrak B}_b$. Alternatively, we may regard $\snc$ as the space
$\prod_i S_i$ ($S_i$ having the discrete topology) which is
homeomorphic to $2^\omega$. These two topologies give the same Borel
structure on $\snc$.

\begin{thm} \label{snthm}
For any $\vec I$, $\vec S$, $\vec n$ as above, $E_{\vec S, \vec n}$
is either smooth, Borel bireducible with $E_0$, Borel bireducible
with $E_1$, or Borel bireducible with $\ell_\infty$.
\end{thm}

\begin{proof}
Suppose $X$, $Y \in \snc$, say $X$ corresponds to the sequence
$(p_i)$ (where $p_i \in S_i$), and $Y$ corresponds to $(q_i)$.
Again the proof of Theorem~\ref{cond}
shows that $X E_{\vec S, \vec n} Y$ iff
$\big(n_i^{| \frac{1}{p_i}- \frac{1}{q_i}|}\big)$ is bounded.
Define $\pi \colon \snc \to \R^\omega$ as follows. If $X$
corresponds to the sequence $(p_i)$, then let
$\pi(X)(i)= \frac{1}{p_i} \log(n_i)$.
Note that all of the $\pi(X)(i)$ take values in the finite set
$B_i:= \{ \frac{1}{p_i} \log(n_i) \colon p_i \in S_i\}$.
We then have that $X E_{\vec S, \vec n} Y$ iff
$\pi(X) E_{\vec{B}} \pi(Y)$. Moreover,
Moreover, $\pi$ is a bijection between $\snc$ and $\prod B_i$.
Thus $E_{\vec{S},\vec{n}}$ is Borel bireducible with $E_{\vec{B}}$.
We are done by Theorem~\ref{bthm}.
\end{proof}

\section{\label{sec:noniso}Nonisomorphic uniformly homeomorphic Banach spaces}

Fix a countable dense set $D \subseteq (2,3)$. If $A$ is a countable
subset of $(2,3)\setminus D$, then we associate to $A$ the separable
Banach space
$$X_A= \left( \sum_{p \in D} \oplus \ell_p \right)_{c_0}
\bigoplus \left( \sum_{q \in A} \oplus \ell_q \right)_{c_0}. $$
Since $(2,3)\setminus D$ is Borel bijectable with $\R$, to prove
Theorem~\ref{thm:main3} it suffices to show that $X_A$ is uniformly
homeomorphic to $X_B$ for any countable $A, B \subseteq
(2,3)\setminus D$, and $X_A$ is not isomorphic to $X_B$ for $A \neq
B$. The following well known lemma, which we sketch a proof for
convenience, verifies the second requirement.

\begin{lem}
If $A \subseteq (2,3)$ is countable and $q \notin A$, then $\ell_q$
is not isomorphic to a subspace of $\left( \sum_{p \in A} \oplus
\ell_p \right)_{c_0}$.
\end{lem}

\begin{proof}
Let $A=\{p_n:n\in\N\}$. For $k\in \N$, denote by $P_k$ the natural
projection onto $\left(\sum_{n=1}^k \oplus \ell_{p_n}\right)_{c_0}$.
Suppose there exists a $X\subset \left(\sum_{n=1}^{\infty} \oplus
\ell_{p_n}\right)_{c_0}$ which is isomorphic to $\ell_q$. Consider
two mutually exclusive cases.

(i) Suppose that there exist $\ep>0$ and $k\in\N$ such that for all
$x\in X$, $\|P_k x\|\ge \ep \|x\|$. Put $X'=P_k(X)$. Then $T:X\to
X'$ defined by $Tx=P_k(x)$, for all $x\in X$, is an isomorphism with
$\|T^{-1}\|\le 1/\ep$. That is, $\ell_q$ is isomorphic to a subspace
of the finite direct sum $\left(\sum_{n=1}^k \oplus
\ell_{p_n}\right)_{c_0}$, which is impossible unless $q=p_n$ for
some $1\le n\le k$.

(ii) Suppose that for all $\ep>0$ and all $k\in\N$ there exists
normalized $x\in X$ with $\|P_k x\|<\ep \|x\|$. Let $(\ep_i)\searrow
0$ such that $\sum_i \ep_i <1/4$. Construct inductively a sequence
of normalized $(x_i)\in X$ and $0<k_1<k_2<k_3\ldots$ such that
$\|x_i-P_{k_i}x_i\|\le \ep_i$ and $\|P_{k_i}x_{i+1}\|<\ep_i$, and
put $x_i'=(P_{k_i}-P_{k_{i-1}})x_i$. Then $(x_i')_{i=1}^{\infty}$ is
a sequence of disjointly supported vectors thus equivalent to the unit vector
basis of $c_0$. Since $\sum_{i=1}^{\infty}\|x_i-x_i'\|\le \sum_i
(\ep_i+\ep_{i-1})<1/2$, this implies that $(x_i)\subset X$ is
equivalent to $c_0$ basis, a contradiction.
\end{proof}

\begin{thm}
Let $D\subset (2,3)$ be dense and $A\subset (2,3)\setminus D$ be
countable. Then $X=\big(\sum_{p\in D}\oplus \ell_p\big)_{c_0}$ is
uniformly homeomorphic to $X_A=\big(\sum_{p\in D}\oplus
\ell_p\big)_{c_0}\oplus\big(\sum_{q\in A}\oplus \ell_q\big)_{c_0}$.
\end{thm}

\begin{remark}
The proof is a slight generalization of Theorem 10.28 in \cite{BL}.
The idea of the proof is due to Ribe \cite{Ri2} who proved it in a
special case. This was later extended by Aharoni and Lindenstrauss
\cite{AL} to a more general setting (see \cite{B} for a nice
exposition). We will reproduce the main steps of the proof following
\cite{BL} with the necessary modifications, and also present an
additional step (Lemma 7.4) clarifying an obscure point there.
\end{remark}

\begin{proof}
Recall that for $x=(x_i)\in\ell_p$, the Mazur map $\varphi_{p,
q}:\ell_p\to\ell_q$ is defined by
$$\varphi_{p, q}(x)=\|x\|_p^{1-\frac{p}{q}}(\textrm{sign}(x_i)|x_i|^{\frac{p}{q}})_i.$$
$\varphi$ is positively homogeneous and, for each $K>0$ it is a
uniform homeomorphism of $K$-ball in $\ell_p$ onto the $K$-ball in
$\ell_q$. Moreover, for every $M$ the family $\{\varphi_{p,q}:1\le
p, q\le M\}$ is a family of equi-uniform homeomorphisms where each
$\varphi_{p,q}$ is restricted to the ball of radius $\exp(1/|p-q|)$
(see Proposition 9.2, \cite{BL}).

Let $(q_j)$ be an enumeration of $A$. Since $D$ is dense, there
exist disjoint infinite subsets $I_j=\{\langle j,n\rangle:n\in\N\}\subset \N$,
$j=1, 2\ldots$, such that $p_{\langle j,n\rangle}\to q_j$ for each $j$. To
simplify the notation, we write $\varphi_{j,n}$ for the Mazur map
$\varphi_{p_{\langle j,n\rangle}, q_j}:\ell_{p_{\langle j,n\rangle}}\to \ell_{q_j}$. By
passing to subsequences of $I_j$'s if necessary, we can and will
assume that the family $\{\varphi_{j,n}, (\varphi_{j,n})^{-1}: j,
n\in\N\}$ of maps where each is restricted to the $2^n$-balls of its
domain is equi-uniformly continuous.

In the next step we solely work on copies of $\ell_{q_j}$'s. The
goal is to construct continuous paths of homeomorphisms between two
particular invertible operators $S_0^j$ and $S_1^j$ described below
in a `uniform' manner. For a fixed $j$, this follows from the fact
that the general linear group of invertible operators on $\ell_q$ is
contractible (cf. e.g., \cite{Mit}). Since we require the paths to
be independent of $j$'s, we give them explicitly.

\begin{lem}\label{lem:path}
There exists a continuous path $\tau\to V_{\tau}$, $0\le \tau\le
1/2$ of invertible operators on $(\ell_q\oplus\ell_q\oplus\ell_q)_q$
such that $V_0$ is the identity and $V_{1/2}(u, v, w)=(u, w, v)$,
for all $(u, v, w)\in (\ell_q\oplus\ell_q\oplus\ell_q)_q$. Moreover,
$\|V_{\tau}\|\le 2$, and the path is independent of $1\le q<\infty$
in the sense that the matrix representation of $V_{\tau}$ with
respect to the decomposition $\ell_q(\ell_q)$ does not depend on
$q$.
\end{lem}

\begin{proof}
Consider an isomorphism $D:\ell_q\to \ell_q(\ell_q)$. $D$ induces an
isomorphism $(\ell_q\oplus\ell_q\oplus\ell_q)_q\to
(\ell_q\oplus\ell_q)_q\oplus(\ell_q\oplus\ell_q\oplus\ldots)_q$
mapping $(u, v, w)$ to $(v, w, (Du)_1, (Du)_2, \ldots)$. Regarding
the latter as a sequence of scalars and composing with obvious
isometries, the operator $V_{1/2}$ can be written as a block
diagonal matrix of the form $V_{1/2}=J\oplus I\oplus I\oplus \ldots$
where $I$ is the $2\times 2$ identity matrix and $J= \left(
\begin{array}{cc}
0 & 1 \\
1 & 0 \end{array} \right)$. For $0\le \tau\le 1/4$, consider the
path $V_{\tau}$ of invertible operators defined by
$V_{\tau}=A_{\tau}\oplus A_{\tau}\oplus A_{\tau}\oplus\ldots$ where
$A_{\tau}= \left(
\begin{array}{cc}
B_{\tau} & C_{\tau} \\
-C_{\tau} & B_{\tau} \end{array} \right)$, and
$$B_{\tau}= \left(
\begin{array}{cc}
\cos^22\pi\tau & \sin^22\pi\tau \\
\sin^22\pi\tau & \cos^22\pi\tau \end{array} \right),\ \  C_{\tau}=
\left(
\begin{array}{cc}
-\cos 2\pi\tau \sin2\pi\tau & \cos 2\pi\tau \sin 2\pi\tau \\
\cos 2\pi\tau \sin 2\pi\tau & -\cos 2\pi\tau \sin 2\pi\tau  \end{array}
\right).$$ 
Thus, the path connects the
identity to $V_{1/4}=J\oplus J\oplus\ldots$.

For $1/4\le \tau\le 1/2$, we continue the path by $V_{\tau}=J\oplus
A_{\tau}\oplus A_{\tau}\oplus\ldots$. Thus $V_{1/2}=J\oplus I\oplus
I\oplus \ldots$, as desired. Note that since $A_{1/4}=J\oplus J$,
two definitions of $V_{1/4}$ coincide and therefore it is well
defined. Clearly, the paths are independent of $1\le q<\infty$, and
an easy computation shows that $\|V_{\tau}\|\le 2$ for all $0\le
\tau\le 1/2$.
\end{proof}

Now for each $j$ let
$T_j:(\ell_{q_j}\oplus\ell_{q_j})_{q_j}\to\ell_{q_j}$ be a linear
isometry, and consider the following isometries from
$(\ell_{q_j}\oplus\ell_{q_j}\oplus\ell_{q_j})_{q_j}$ onto
$(\ell_{q_j}\oplus\ell_{q_j})_{q_j}$ induced by $T_j$'s:
$$\begin{array}{rcl}
S_0^j(u, v, w)&=&(T_j(u, v), w)\ \textrm{ and}\\
 S_1^j (u, v, w)&=&(v, T_j(u,w))\ \textrm{ for }\
(u,v,w)\in (\ell_{q_j}\oplus\ell_{q_j}\oplus\ell_{q_j})_{q_j}.
\end{array}$$

\begin{lem}\label{lem:paths}
For all $j$ and $0\le \tau\le 1$, there is a homogeneous
norm-preserving homeomorphism
$h_{\tau}^j:(\ell_{q_j}\oplus\ell_{q_j}\oplus\ell_{q_j})_{q_j}\to
(\ell_{q_j}\oplus\ell_{q_j})_{q_j}$ such that $h_0^j=S_0^j$ and
$h_1^j=S_1^j$, and such that there is a constant $K$ (independent of
$j$) for which
$$\|h_{\tau}^j(x)-h_{\eta}^j(y)\|\le K\big(\|x-y\|+|\tau-\eta|\max(\|x\|,
\|y\|)\big)$$ and similarly for their inverses.
\end{lem}

\begin{proof} For all $j$ and $0\le \tau\le 1/2$, let $V^j_\tau$ be given by Lemma \ref{lem:path} for $q_j$.
We define $S^j_\tau$ for all $j$ and $0\leq \tau\leq 1$ as follows. For $0\le \tau \le 1/2$, define $S_{\tau}^j(u, v,
w)=(T_j(u_{\tau},v_{\tau}), w_{\tau})$, where $(u_\tau, v_\tau,
w_\tau)=V_\tau^j(u,v,w)$. For $1/2\le \tau\le 1$, put
$S_{\tau}^j=U_{\tau}^jS_{1/2}^j$ where $U_{\tau}^j$ is a path of
invertible operators on $(\ell_{q_j}\oplus\ell_{q_j})_{q_j}$
connecting the identity to the operator $(u, v)\to (v, u)$. To get
the path $U_{\tau}^j$, start with the isomorphism
$E^j:(\ell_{q_j}\oplus\ell_{q_j})_{q_j}\to
(\ell_{q_j}\oplus\ell_{q_j}\oplus\ldots)_{q_j}$ defined by
$E^j(u,v)=((D^ju)_1, (D^jv)_1, (D^ju)_2, (D^jv)_2, \ldots)$ where
$D^j:\ell_{q_j}\to \ell_{q_j}(\ell_{q_j})$ is an isomorphism. Then
put $U_{\tau}^j=(E^j)^{-1}\tilde{V}_{\tau}^j E^j$ where
$\tilde{V}_{\tau}^j= V^j_{(2\tau-1)/4}$.
Note that the norm of $S_{\tau}^j$'s and their
inverses are uniformly bounded, and it is clear from the formulas
that $S_{\tau}^j$'s are Lipschitz in $\tau$. Finally, putting
$h_{\tau}^j(x)=\|x\|S_{\tau}^j(x)/\|S_{\tau}^j(x)\|$ yields the
desired norm-preserving maps. Note that the inverses of the
normalized maps have the same form, that is,
$(h_{\tau}^j)^{-1}(y)=\|y\|(S_{\tau}^j)^{-1}(y)/\|(S_{\tau}^j)^{-1}(y)\|$.
\end{proof}

The desired uniform homeomorphism from $X_A$ onto $X$ will be
defined by $\phi(x)=g_{\|x\|}(x)$ where $g_t(x)=\|x\|\tilde
g_t(x)/\|\tilde g_t(x)\|$, $t>0$ and $\tilde g_t$ is defined below.

Every $x\in X_A$ has a unique representation of the form $\sum_j
u_j+\sum_i x_i$ where $u_j\in \ell_{q_j}$, $q_j\in A$ and $x_i\in
\ell_{p_i}$, $p_i\in D$. For notational convenience we split the
second sum and write this as
$$x=\sum_{j=1}^{\infty}\big(u_j+\sum_{n}x_{j,n}) +\sum_{i\in
I_0}x_i$$ where $x_{j,n}\in \ell_{p_{\langle j,n\rangle}}$ and $I_0$ is the set
of indices which does not belong to any $I_j$'s, $j=1, 2, \ldots$.

Using the Mazur maps we define
$\psi_{j,n}:\big(\ell_{p_{\langle j,n\rangle}}\oplus
\ell_{p_{\langle j,n+1\rangle}}\big)_{c_0}\to (\ell_{q_j}\oplus\ell_{q_j})_{q_j}$
by
$$\psi_{j,n}(x_{j,n},x_{j,n+1})=\big(\varphi_{j,n}(x_{j,n}),
\varphi_{j,n+1}(x_{j,n+1})\big).$$ Then $\psi_{j,n}$'s are
equi-uniform homeomorphisms between $2^n$-balls of the domain and
$2^{1/q_j}2^n$-balls of the range. Let $\omega(\ep)$ denote the
common bound of moduli of continuity of $\varphi_{j,n}$ and
$\psi_{j,n}$ and of their inverses.

For $n=0,1,\ldots$ put $\alpha_n=2^n -1$. For $\alpha_n\le t\le
\alpha_{n+1}$ define
$$\tilde
g_t(x)=\sum_{j=1}^{\infty}\Big[\psi_{j,n}^{-1}\left(h_{2^{-n}(t-\alpha_n)}(u_j,\psi_{j,n}(x_{j,n},
x_{j,n+1}))\right)+\sum_{i\neq n, n+1}x_{j,i}\Big]+\sum_{i\in
I_0}x_i.$$

Here for each $j$, the map only replaces three coordinates in the
block $(u_j, \ldots, x_{j, n}, x_{j, n+1},\ldots)$ by $(\ldots,
y_{j,n}, y_{j,n+1}, \ldots)$ where $(y_{j,n},
y_{j,n+1})=\psi_{j,n}^{-1}\left(h_{2^{-n}(t-\alpha_n)}(u_j,\psi_{j,n}(x_{j,n},
x_{j,n+1}))\right)$. Note that for $t=\alpha_n$, $\tilde g_t$ is
defined twice, however, the two definitions using $h_1^j(u_j,
\psi_{j, n-1}(x_{j, n-1}, x_{j,n}))$ and the one using $h_0^j(u_j,
\psi_{j, n}(x_{j, n}, x_{j,n+1}))$ coincide, therefore it is
well-defined.

It is clear from the definition that $\tilde g_t$'s are homogeneous.
It remains to check $\{\tilde g_t: 0\le t<\infty\}$ is a family of
equi-uniform homeomorphisms. This will imply the same for the
norm-preserving $g_t$'s (see the remarks before Theorem 10.28,
\cite{BL}).

Let $x=\sum_{j=1}^{\infty}\big(u_j+\sum_{n}x_{j,n}) +\sum_{i\in
I_0}x_i$ and $y=\sum_{j=1}^{\infty}\big(v_j+\sum_{n}y_{j,n})
+\sum_{i\in I_0}y_i$ be in $X_A$ such that $\|x\|, \|y\|\le
\alpha_{n+1}$ and $\|x-y\|<1$, and let $\alpha_n\le t, s\le
\alpha_{n+1}$. Then $\|\tilde g_t(x)-\tilde g_s(y)\|$ is bounded by
$$
\Big\|\displaystyle\sum_{j}
\left[\psi_{j,n}^{-1}\Big(h_{2^{-n}(t-\alpha_n)}(u_j,\psi_{j,n}(x_{j,n},
x_{j,n+1}))\Big)-
\psi_{j,n}^{-1}\Big(h_{2^{-n}(s-\alpha_n)}(v_j,\psi_{j,n}(y_{j,n},
y_{j,n+1}))\Big) \right]\Big\|_{c_0}
$$
$$
+\ \Big\| \sum_j \sum_{i\neq n,n+1}(x_{j,i}-y_{j,i})+\sum_{i\in I_0}(x_i-y_i)\Big\|_{c_0}.\ \
\ \ \ \ \ \ \ \ \ \ \ \ \ \  \ \ \  \ \ \ \  \ \ \  \ \ \ \ \ \ \ \ \ \ \ \ \ \ \ \ \ \ \ \ \
\ \ \ \
$$
The second term is bounded by $\|x-y\|$. By Lemma \ref{lem:paths},
for all $j$,
$$\left\|\psi_{j,n}^{-1}\Big(h_{2^{-n}(t-\alpha_n)}(u_j,\psi_{j,n}(x_{j,n},
x_{j,n+1}))\Big)-
\psi_{j,n}^{-1}\Big(h_{2^{-n}(s-\alpha_n)}(v_j,\psi_{j,n}(y_{j,n},
y_{j,n+1}))\Big)\right\|$$ is bounded by
$$\omega\big(K\big\{\|u_j-v_j\|+\omega(\|x_{j,n}-y_{j,n}\|+\|x_{j,
n+1}-y_{j,n+1}\|)\big\}+\alpha_{n+1}|2^{-n}s-2^{-n}t|\big),$$ where
the constant $K$ and the function $\omega$ is independent of $j$.
Since $\ep\le C\omega(\ep)$ for some constant $C$ and all $\ep\le
1$, and since the estimates are independent of $j$'s, it follows that
there exist constants $L_1$ and $L_2$ such that
$$\|\tilde g_t(x)-\tilde g_s(y)\|\le L_1\omega\big (L_2
\omega(\|x-y\|)+|s-t|\big).$$ The same estimates hold for the
inverses as well.
\end{proof}

\end{document}